\documentclass{amsart}[11pt,amssymb]

\usepackage{amsmath,amssymb,amsthm,amsfonts,setspace,hyperref}

\DeclareMathSymbol{\twoheadrightarrow} {\mathrel}{AMSa}{"10}

\def\Q{{\mathbf Q}}

        \def\CC{{\mathfrak C}}

\def\Z{{\mathbf Z}}

\def\C{{\mathbf C}}

\def\H{{\mathrm H}}

\def\ST{{\mathbf S}}

\def\Sn{{\mathbf S}_n}

\def\An{{\mathbf A}_n}

\def\Gal{\mathrm{Gal}}

\def\Lie{\mathrm{Lie}}

\def\MT{\mathrm{MT}}

\def\mt{\mathrm{mt}}

                      \def\Hdg{\mathrm{Hdg}}

                      \def\hdg{\mathrm{hdg}}

\def\End{\mathrm{End}}

\def\Aut{\mathrm{Aut}}

\def\Hom{\mathrm{Hom}}

\def\I{\mathrm{Id}}
                                    \def\TT{\mathrm{T}}

\def\GL{\mathrm{GL}}

                                                \def\U{\mathrm{U}}

\def\SL{\mathrm{SL}}

                                    \def\Tr{\mathrm{Tr}}

                \def\sL{\mathfrak{sl}}

        \def\K_a{\bar{K}}

\def\E{\mathrm{E}}

\def\dim{\mathrm{dim}}

\def\c{{\mathfrak c}}

\def\q{{\mathfrak q}}

\def\G{{\mathbf G}}

\def\K{{\mathcal{K}}}

\def\m{{\mathfrak m}}

\def\ZZ{{\mathfrak Z}}

                           \def\f{{\mathfrak f}}

                           \def\cc{{\mathfrak c}}

                           \def\h{{\mathbf h}}

\def\E{{\mathcal E}}

\newcommand{\ra}{\rightarrow}

\newcommand{\gp}{\left(\Z/q\Z\right)^{\times}}

\newcommand{\xs}{\chi^{\scriptscriptstyle{\bigstar}}}

\newcommand{\bxs}{\overline{\chi}^{\scriptscriptstyle{\bigstar}}}

\newcommand{\abs}[1]{\lvert#1\rvert}

\newtheorem{thm}{Theorem}[section]

\newtheorem{lem}[thm]{Lemma}

\newtheorem{cor}[thm]{Corollary}

\newtheorem{prop}[thm]{Proposition}

\theoremstyle{definition}

\newtheorem{defn}[thm]{Definition}

\newtheorem{ex}[thm]{Example}

\newtheorem{exs}[thm]{Examples}

\newtheorem{rem}[thm]{Remark}

\newtheorem{rems}[thm]{Remarks}

\newtheorem{sect}[thm]{}

\title[Hodge groups of superelliptic jacobians]{Centers of Hodge groups of superelliptic jacobians}

\author{Jiangwei Xue}

\address{Department of Mathematics, Pennsylvania
State University, University Park, PA 16802, USA}

\email{xue\_j\char`\@math.psu.edu}

\author{Yuri G. Zarhin}

\address{Department of Mathematics, Pennsylvania
State University, University Park, PA 16802, USA}

\email{zarhin\char`\@math.psu.edu}

\begin{document}

\maketitle

\section{Introduction}
\label{intro}

Let $\C$ be the field of complex numbers. If $z\in \C$ then we write $\bar{z}$
for its complex-conjugate and denote by $\iota: \C \to \C$ the corresponding
element of the group $\Aut(\C)$ of automorphisms of $\C$. We write
$\bar{\Q}\subset\C$ for the algebraic closure of $\Q$ in $\C$. It is well-known
that the subfield $\bar{\Q}$ is $\Aut(\C)$-stable and the natural homomorphism
$$\Aut(\C)\to \Gal(\bar{\Q}/\Q)$$
is surjective. If $W$ is a $\Q$-vector space, $\Q$-algebra or $\Q$-Lie algebra
then we write $W_{\C}$ for the corresponding $\C$-vector space (respectively,
$\C$-algebra or $\C$-Lie algebra) $W\otimes_{\Q}\C$.

Let $f(x)\in \C[x]$ be a polynomial of degree $n\ge 2$  without multiple roots.
 Suppose that  $p$ is a prime that does not divide $n$ and
 a positive integer $q=p^r$ is a power of $p$. As usual, $\varphi(q)=(p-1)p^{r-1}$ denotes the Euler function.
Let us fix a primitive $q$th root of unity
 $\zeta_q\in\C$.
 We write $C_{f,q}$
for the superelliptic curve $y^q=f(x)$ and $J(C_{f,q})$ for its jacobian.
Clearly, $J(C_{f,q})$ is an abelian variety  and
$$\dim(J(C_{f,q}))=\frac{(n-1)(q-1)}{2}.$$
The periodic automorphism $(x,y)\mapsto (x,\zeta_q y)$ of $C_{f,q}$ induces by
Albanese functoriality the periodic automorphism of $J(C_{f,q})$ that we denote
by $\delta_q$. It is known \cite{SPoonen, ZarhinM} that $\delta_q$ gives rise
to an embedding of the product $\prod_{i=1}^r\Q(\zeta_{p^i})$ of cyclotomic
fields into the endomorphism algebra $\End^0(J(C_{f,q}))$ of $J(C_{f,q})$. (If
$q=p$ then we actually get an embedding $\Z[\delta_p] \hookrightarrow
\End(J(C_{f,p}))$ that sends $\zeta_p$ to $\delta_p$.) More precisely, if $q\ne
p$ then the map $(x,y) \to (x,y^p)$ defines the map of curves $C_{f,q}\to
C_{f,q/p}$, which induces (by Albanese functoriality) the surjective
homomorphism $J(C_{f,q})\to J(C_{f,q/p})$
 of complex abelian varieties; we write $J^{(f,q)}$ for the identity component of its kernel.
 (If $q=p$ then we put $J^{(f,p)}=J(C_{f,p})$.) One may check \cite{ZarhinM}  that
  $J(C_{f,q})$ is isogenous to the product $\prod_{i=1}^r J^{(f,p^i)}$ and $\delta_q$ gives rise to an embedding
  $$\Z[\zeta_q]\hookrightarrow \End(J^{(f,q)}).$$

 In a series of papers \cite{ZarhinMRL,ZarhinCrelle,ZarhinCamb,ZarhinM,ZarhinMZ2}, one of the authors (Y.Z.) was able
to prove that
$$\End(J^{(f,q)})=\Z[\zeta_q], \ \End^{0}(J^{(f,q)})=\Q(\zeta_q)$$
 assuming that $n \ge 5$ and there exists a subfield $K\subset\C$ such that all
 the coefficients of $f(x)$ lie in $K$ and the Galois group of $f(x)$ over $K$
 is either the full symmetric group $\Sn$ or the  alternating group $\An$. In
 particular, $\End^0(J(C_{f,q})) \cong \prod_{i=1}^r\Q(\zeta_{p^i})$. (The same
 assertion holds true if $n=4$, the prime  $p$ is odd, $\zeta_q \in K$ and the Galois group is $\ST_4$.)

Our goal is to study the  (reductive $\Q$-algebraic connected) Hodge group
$\Hdg=\Hdg(J^{(f,q)})$ of $J^{(f,q)}$. Notice that when $q=2$ (i.e., in the
hyperelliptic case) this group was completely determined in \cite{ZarhinMMJ}
(when $f(x)$ has ``large" Galois group); in particular, in this case the Hodge
group is simple and the center of its Lie algebra is $\{0\}$. So, further we
assume that $q>2$ and therefore $\Q(\zeta_q)$ is a CM-field. So, if
$\End^{0}(J^{(f,q)})=\Q(\zeta_q)$ then (see Remark \ref{centerhdg} below) the
center $\cc^{0}$ of the $\Q$-Lie algebra $\hdg$ of $\Hdg(J^{(f,q)})$ lies in
$$\Q(\zeta_q)_{-}:=\{e\in \Q(\zeta_q)\mid \bar{e}=-e\}\subset \Q(\zeta_q).$$
(If $q=2$ then $\Q(\zeta_2)=\Q$ and  $\Q(\zeta_2)_{-}=\{0\}$.) In particular,
its dimension does not exceed $\varphi(q)/2$; the equality holds if and only if
$q>2$ and $\cc^{0}$ coincides with $\Q(\zeta_q)_{-}$.

Let
$$\Q(\zeta_q)^{+}:=\{e\in \Q(\zeta_q)\mid \bar{e}=e\}\subset \Q(\zeta_q)$$
be the maximal totally real subfield of $§\Q(\zeta_q)$. If $q>2$ then
$[\Q(\zeta_q)^{+}:\Q]=\varphi(q)/2$. We write $R_{\Q(\zeta_q)}\G_m$ and
$R_{\Q(\zeta_q)^{+}}\G_m$ for the algebraic $\Q$-tori obtained by the Weil
restriction of scalars of the multiplicative group $\G_m$ to $\Q$ from
$\Q(\zeta_q)$ and $\Q(\zeta_q)^{+}$ respectively. The norm map $\Q(\zeta_q)
\to\Q(\zeta_q)^{+}$ induces the natural homomorphism of algebraic $\Q$-tori and
 we denote by $\U_q=\TT_{\Q(\zeta_q)}$ its kernel, i.e., the corresponding
{\sl norm torus} \cite{Voskresenskii}.  It is well known that $\U_q$ is an
algebraic $\Q$-torus (in particular, it is connected) and
$$\U_q(\Q)=\{e\in \Q(\zeta_q)\mid \bar{e}e=1\}\subset \Q(\zeta_q).$$  The embedding
$\Q(\zeta_q) \hookrightarrow \End^{0}(J^{(f,q)})$ allows us to identify
$\Q(\zeta_q)$ with a certain $\Q$-subalgebra of $\End_{\Q}(\H^1(J^{(f,q)}))$ and
consider $R_{\Q(\zeta_q)}\G_m$ and therefore $\U_q$ as certain $\Q$-algebraic
subgroups of the general linear group $\GL(\H^1(J^{(f,q)},\Q))$ over $\Q$. Then
the $\Q$-Lie algebras of $R_{\Q(\zeta_q)}\G_m$ and $\U_q$, viewed as $\Q$-Lie
subalgebras of $\End_{\Q}(\H^1(J^{(f,q)}))$, coincide with $\Q(\zeta_q)$ and
$\Q(\zeta_q)_{-}$ respectively.

Recall that $J^{(f,q)}$ is an abelian subvariety of the jacobian $J(C_{f,q})$
and consider the ($\delta_q$-invariant) polarization $\lambda_r$ on $J^{(f,q)}$
induced by the canonical principal polarization on $J(C_{f,q})$. The
polarization $\lambda_r$ gives rise to a certain $\delta_q$-invariant
nondegenerate
  alternating
  $\Q$-bilinear form
  $$\psi_r:\H_1(J^{(f,q)},\Q) \times \H_1(J^{(f,q)},\Q) \to \Q$$
   (This form is the imaginary part of the {\sl Riemann form} of
  $\lambda_r$ \cite{Mumford,Ribet3}.) The $\delta_q$-invariance implies that
  $\psi_r(ex,y)=\psi_r(x,\bar{e}y) \ \forall e \in \Q(\zeta_q); \ x,y\in
  \H_1(J^{(f,q)},\Q)$.
  If $q>2$ then we choose a nonzero element $\beta_r \in \Q(\zeta_q)_{-}$  and a standard
construction (see, for instance, \cite[p. 531]{Ribet3}) gives us
   a nondegenerate Hermitian $\Q(\zeta_q)$-sesquilinear form
  $$\phi_r: \H_1(J^{(f,q)},\Q) \times \H_1(J^{(f,q)},\Q) \to \Q(\zeta_q)$$
  such that $\Tr_{\Q(\zeta_q))/\Q}(\beta_r\phi_r)=\psi_r$. We write $\U(\H_1(J^{(f,q)}),\Q),\phi_r)$
  for the unitary
group of $\phi_r$,
viewed as an algebraic (reductive) $\Q$-subgroup of $\GL(\H_1(J^{(f,p^i)},\Q))$
(via Weil's restriction of scalars from $\Q(\zeta_q)^{+}$ to $\Q$ (ibid). Then
the center of $\U(\H_1(J^{(f,q)}),\Q),\phi_r)$ coincides with $\U_q$.

Since the Hodge group of $J^{(f,q)}$ respects the polarization and commutes
with endomorphisms of $J^{(f,q)}$,
$$\Hdg(J^{(f,q)})\subset \U(\H_1(J^{(f,q)},\Q),\phi_r).$$
Recall that the centralizer of $\Hdg(J^{(f,q)})$ in
$\End_{\Q}(\H_1(J^{(f,q)},\Q))$ coincides with $\End^0(J^{(f,q)})$. This
implies that if $\End^0(J^{(f,q)})$ coincides with $\Q(\zeta_q)$ then
  the center of $\Hdg(J^{(f,q)})$ lies in $\U_q$.


\begin{rem}
  Let $\Hdg^{ss}=[\Hdg,\Hdg]$ be the derived subgroup of $\Hdg$. Let
  $\ZZ$ be the center of $\Hdg$ and $\ZZ^0$ the identity component of
  $\ZZ$.  Since the Hodge group is connected reductive, $\Hdg^{ss}$
  is a semisimple connected algebraic $\Q$-group, $\ZZ^0$ an
  algebraic $\Q$-torus and the natural morphism of linear algebraic
  $\Q$-groups $\Hdg^{ss} \times \ZZ^0\to \Hdg$ is an isogeny. It
  follows that the $\Q$-Lie algebra $\Lie(\ZZ)$ of $\ZZ$ coincides
  with the $\Q$-Lie algebra $\Lie(\ZZ^{0})$ of $\ZZ^{0}$ and equals
  $\c^{0}$.
\end{rem}

\begin{thm}
\label{main0} Assume that  $n\ge 3$ and  $p$ does not divide $n$. Let $f(x)\in
\C[x]$ be a degree $n$ polynomial without multiple roots. If $q>2$ then the
center $\cc^{0}$ of  the $\Q$-Lie algebra $\hdg$ of $\Hdg(J^{(f,q)})$
 has $\Q$-dimension greater or equal than $\varphi(q)/2$. In other
 words, the center of $\Hdg(J^{(f,q)})$ has dimension greater or equal
 than $\varphi(q)/2$.
\end{thm}

As an application, we obtain the following  statement.

\begin{thm}
\label{main} Assume that  $n\ge 4$ and   $p$ does not divide $n$.
  Let $K$ be a subfield of $\C$
 that contains all the coefficients of
$f(x)$. Suppose that $f(x)$ is irreducible over $K$ and the Galois group
$\Gal(f)$ of $f(x)$ over $K$ is either $\Sn$ or  $\An$. Assume additionally
that either $n \ge 5$ or $n=4$ and $\Gal(f)=\ST_4$.

If $q>2$ then the center $\cc^{0}$ of the $\Q$-Lie algebra $\hdg$ of
$\Hdg(J^{(f,q)})$ has $\Q$-dimension $\varphi(q)/2$ and coincides with
$\Q(\zeta_q)_{-}$. In addition, the center of $\Hdg(J^{(f,q)})$ coincides with
$\U_q$.
\end{thm}

\begin{ex}
Suppose that $n, p, f(x)$ enjoy the conditions of Theorem \ref{main}. Assume
additionally that $p$ is odd.
 Since $J(C_{f,p})=J^{(f,p)}$, we conclude that
the center of $\Hdg(J(C_{f,p}))$ coincides with
$\U_p$.
\end{ex}


\begin{rem}
In Theorem \ref{main} we prove that the center of the Hodge group of
$J^{(f,q)}$ is ``as large as possible", taking into account that the
endomorphism algebra of $J^{(f,q)}$ coincides with $\Q(\zeta_q)$. In fact, our
goal was (and still is) to prove that (under the assumptions of Theorem
\ref{main}) the whole Hodge group is ``as large as possible", i.e., coincides
with $\U(\H_1(J^{(f,q)},\Q),\phi_r)$,
  which would imply that all Hodge classes on  each self-product of $J^{(f,q)}$ can be
  presented as  linear combinations of products of divisor classes and, in particular, the
  validity of the Hodge conjecture  for all the self-products \cite[p. 528 and 531]{Ribet3}. Since the Hodge
  group is connected reductive, the problem splits naturally in two parts:
  to prove that the center of $\Hdg(J^{(f,q)})$ is ``as large as possible" (i.e., coincides with $\U_q$) and
  that the derived subgroup (semisimple part) of $\Hdg(J^{(f,q)})$ is ``as large as possible" (i.e., coincides with
  the corresponding special unitary group). Theorem
  \ref{main} settles the first one. (The second problem is solved in \cite{XZ2}
  under certain additional conditions on $n$ and $q$.)
\end{rem}

 In order to describe our results for the whole  $J(C_{f,q})$ when $q>p$, let us  put
$$E^{p,i}:=\Q(\zeta_{p^i}), \ E^{p,i}_{-} := \Q(\zeta_{p^i})_{-},$$
$$\E^{p,r}_{-} :=\{ (e_i)_{i=1}^r  \in \oplus_{i=1}^r E^{p,i}_{-}\mid
\Tr_{E^{p,i+1}/E^{p,i}}(e_{i+1})=e_i \ \forall i<r\}\subset\oplus_{i=1}^r
E^{p,i}_{-} .$$
\begin{thm}
\label{mainfull} Assume that  $n\ge 4$ and   $p$ does not divide $n$. 
Let $K$ be a subfield of $\C$
 that contains all the coefficients of
$f(x)$. 
Suppose that $f(x)$ is irreducible over $K$ and the Galois group
$\Gal(f)$ of $f(x)$ over $K$ is either $\Sn$ or  $\An$. Assume additionally
that either $n \ge 5$ or $n=4$ and $\Gal(f)=\ST_4$. Let us consider the abelian
variety $Z=\prod_{i=1}^r J^{(f,p^i)}$ and its first rational homology group
$\H_1(Z,\Q)=\oplus_{i=1}^r \H_1(J^{(f,p^i)},\Q)$. If $p^r>2$ then the center
$\cc^{0}_Z$ of the $\Q$-Lie algebra $\hdg_Z$ of the Hodge group $\Hdg(Z)$ of
$Z$ has $\Q$-dimension $\varphi(p^r)/2$ and coincides with
$$\E^{p,r}_{-} \subset \oplus_{i=1}^r  E^{p,i}_{-}\subset \oplus_{i=1}^r
\Q(\zeta_{p^i})\subset \oplus_{i=1}^r \End_{\Q}(\H_1(J^{(f,p^i)},\Q))\subset
\End_{\Q}(\H_1(Z,\Q)).$$
\end{thm}

\begin{rem}
\label{isog} Let us fix an isogeny $\alpha: J(C_{f,p^r})\to \prod_{i=1}^r
J^{(f,p^i)}=Z$. Then $\alpha$ induces an isomorphism of $\Q$-vector spaces
$\alpha: \H_1(J(C_{f,p^r}),\Q) \cong \H_1(Z,\Q)$. Clearly, the Hodge group of
$J(C_{f,p^r})$ coincides with $\alpha^{-1}\Hdg(Z)\alpha$. This implies that if
 $q>2$ then  the center  of the $\Q$-Lie algebra of
$\Hdg(J(C_{f,p^r}))$ has $\Q$-dimension $\varphi(p^r)/2$ and coincides with
$\alpha^{-1}\E^{p,r}_{-} \alpha$.
\end{rem}
\begin{rem} We keep the notation and assumptions of  Theorem \ref{mainfull} and Remark
\ref{isog}.
  Let us identify (via $\alpha$) $\H_1(J(C_{f,q}),\Q)$ with
  $\oplus_{i=1}^r \H_1(J^{(f,p^i)},\Q)$.
Since the Hodge group of $J(C_{f,q})$ respects the polarization and commutes
with endomorphisms of $J(C_{f,q})$,
$$\Hdg(J(C_{f,q}))\subset \prod _{i=1}^r \Hdg(J^{(f,p^i)}) \subset \prod_{i=1}^r \U(\H_1(J^{(f,p^i)},\Q),\phi_i).$$
  Let $G$ be the reductive
  $\Q$-algebraic subgroup of $\GL(\H_1(J(C_{f,q}),\Q))$ that is \textit{cut
  out} by the polarization and the endomorphisms of $J(C_{f,q})$ \cite[p. 528]{Ribet3}).

  Now assume that $p$ is {\sl odd}. Taking
  into account that
  all $\End^0(J^{(f,p^i)})$ are (mutually nonisomorphic) CM-fields $E^{p,i}$ and using results from p. 531
   of \cite{Ribet},  one may
  easily check that
  $G=\prod_{i=1}^r \U(\H_1(J^{(f,p^i)},\Q),\phi_i)$.
  It follows  that the center of the $\Q$-Lie algebra of $G$ coincides
  with $\oplus_{i=1}^r E_{-}^{p,i}$. On the other hand,
  Theorem \ref{mainfull} and Remark \ref{isog} imply that (under their
  assumptions) the center $\cc^{0}$ of the $\Q$-Lie algebra  of
  $\Hdg(J(C_{f,q}))$ is the {\sl proper} subspace $\E^{p,r}_{-}$ of $\oplus_{i=1}^r E_{-}^{p,i}$. It follows
  that $\Hdg(J(C_{f,q})) \ne G$ and therefore
   $\Hdg(J(C_{f,q}))$ is a {\sl proper}
  subgroup of $G$.  This implies that a certain self-product of $J(C_{f,q})$
  admits an {\sl exotic} Hodge class that could not be presented as a
  linear combinations of products of divisor classes. The same assertion holds
  true if $p=2$ and $r \ge 3$.
\end{rem}

Another application of Theorem \ref{main0} is the following statement.

\begin{thm}
  \label{mainD} Assume that $n\ge 3$ and $p$ does not divide $n$.  Let
  $f(x)\in \C[x]$ be a degree $n$ polynomial without multiple
  roots. Assume also that $q>2$.

\begin{itemize}
\item[(i)] If $p$ is odd then $J^{(f,q)}$ contains a simple complex abelian
subvariety $T$ with \[\dim(T)\ge \varphi((p-1)p^{r-1})\geq \varphi(p-1).\] In
particular, $\dim(T)\ge \varphi(p-1)\cdot (p-1)p^{r-2}$ when $r\geq 2$.

\item[(ii)] If $p=2$ and $r\geq 3$ then $J^{(f,q)}$ contains a simple complex
abelian subvariety $T$ with $\dim(T)\ge 2^{r-3}$.
\end{itemize}
\end{thm}

\begin{rem}
  Actually, our proof gives a little bit more, namely, that the center
  $\CC_T$ of $\End^0(T)$ is a CM-field such that $[\CC_T:\Q]/2$ is
  greater or equal than the lower bound given in Theorem
  \ref{mainD}. (Notice that $\CC_T$ is a direct summand of the center
  of $\End(J^{(f,q)})$.)
\end{rem}

\begin{cor}[Corollary to Theorem \ref{mainD}]
\label{elliptic}
 Suppose that $n \ge 3$ and  $d$ is a positive integer such that $(d,n)=1$.
Let $f(x)\in \C[x]$ be a degree $n$ polynomial without multiple roots. Assume
that $d \ge 5$ and $d$ is neither $6$ nor $8$ nor $12$ nor $24$. Let us consider the
superelliptic curve $C_{f,d}:y^d=f(x)$ and let $J(C_{f,d})$ be its jacobian.

Then $J(C_{f,d})$ is not isogenous to a product of elliptic curves.
\end{cor}

\begin{proof}
Clearly, $d$ has a divisor $q$ such that either $q$ is a prime $\ge 5$ or $q=9$
or $q=16$. The existence of the covering of algebraic
curves
$$C_{f,d} \to C_{f,q}, \ (x,y) \mapsto (x, y^{d/q})$$
implies that $J(C_{f,d})$ has a quotient isomorphic to $J(C_{f,q})$. Now the
result follows from Theorem \ref{mainD} if we take into account that
$J^{(f,q)}$ is an abelian subvariety of $J(C_{f,q})$.
\end{proof}

\begin{rem}
Corollary \ref{elliptic} implies that if $p\ge 5$ then none of jacobians of
$C_{f,p}$ is  {\sl totally split} in a sense of \cite{ES}. The same is true for the 
jacobians of $C_{f,16}$ and $C_{f,9}$.
\end{rem}

\begin{rem}
Recently D. Ulmer \cite{Ulmer}, using a construction of L. Berger
\cite{Berger}, found out that  the rank of  the Mordell-Weil group of the
jacobian of the curve $f(x)-t f(y)=0$ over the function field $\C(t^{1/q})$ is
closely related to the
 endomorphism algebras of $J^{(f,p^i)}$ (for $i\le r$).
One may hope that our results could be useful for the study of the rank of
abelian varieties in infinite towers of function fields.
\end{rem}

The paper is organized as follows. In Section \ref{field} we discuss auxiliary
results related to CM-fields. Section \ref{MT} treats complex abelian varieties
with multiplication by CM-fields. Section \ref{sectmain} contains the proof of
main results modulo some arithmetic properties of certain (non-vanishing)
Fourier coefficients with respect to the finite commutative group
$(\Z/q\Z)^{*}$; those properties are proved in Sections \ref{arithm} and
\ref{lastsection}. Last section contains an auxiliary result from semilinear
algebra.

\smallskip

{\bf Acknowledgments}. We are grateful to Professor R. Vaughan for useful
discussions.
The final version of this paper was written during the IAS/Park City  summer school
 ``Arithmetic of $L$-functons".
We are grateful to the organizers for the invitations and to the PCMI for its  hospitality and support.

\section{Field embeddings}
\label{field}
\begin{sect}

If $V$ is a finite-dimensional $\Q$-vector space (resp. $\Q$-algebra) then we
write $V_{\C}$ for the correspondind finite-dimensional $\C$-vector space
(resp. $\C$-algebra) $V\otimes_{\Q}\C$; clearly
$$\dim_{\Q}(V)=\dim_{\C}(V_{\C}).$$
The group $\Aut(\C)$ acts tautologically on $\C$ and the subfield of
$\Aut(\C)$-invariants coincides with $\Q$. This allows us to define the {\sl
tautological} semilinear action on $V_{\C}$ as follows.
$$s(v\otimes z)=v\otimes s(z) \ \forall \ s \in \Aut(\C), v\in V, z\in \C.$$
The semilinearity means that
$$s(zv)=s(z)s(v) \ \forall \ s \in \Aut(\C), v\in V_{\C}, z\in \C.$$
 Clearly, the $\Q$-subspace of all
$\Aut(\C)$-invariant elements in $V_{\C}$ coincides with $V\otimes 1=V$. It is
also clear that if $W\subset V$ is a $\Q$-vector subspace then $W_{\C}$ is a
$\Aut(\C)$-stable complex vector subspace in $V_{\C}$. Conversely, if
$\tilde{W}$ is a $\Aut(\C)$-stable complex vector subspace in $V_{\C}$ then
there exists exactly one $\Q$-vector subspace $W\subset V$ such that
$\tilde{W}=W_{\C}$; in addition, $W$ is the $\Q$-vector subspace of all
$\Aut(\C)$-invariant elements in $\tilde{W}$. (See Sect. \ref{semilinear}.)
\end{sect}

\begin{sect}

 Let $E$ be a number field. Let $\Sigma_E$ be the set of all field
embeddings $\sigma: E \hookrightarrow \C$. Clearly, $\sigma(E)\subset \bar{\Q}$
for all $\sigma$. It is well-known that $\Sigma_E$ consists of $[E:\Q]$
elements. The group $\Aut(\C)$  acts naturally on $\Sigma_E$. Namely, if
$\sigma:E\hookrightarrow \C$ is a field embedding and $s$ is an automorphism of
$\C$ then we define $s(\sigma):E\hookrightarrow \C$ as the composition
$$s\sigma: E \ \hookrightarrow \C \to \C.$$
If $\sigma\in \Sigma_E$ then we write $\bar{\sigma}$ for the complex-conjugate
of $\sigma$, i.e., for the composition $\iota\sigma: E \hookrightarrow \C$.
Clearly, the action of $\Aut(\C)$ on $\Sigma_E$ factors through the natural
surjection $\Aut(\C)\twoheadrightarrow \Gal(\bar{\Q}/\Q)$.

 Let us consider the $[E:\Q]$-dimensional $\C$-algebra $\C^{\Sigma_E}$
of all functions $\phi: \Sigma_E \to \C$. The action of $\Aut(\C)$ induces the
semilinear action of $\Aut(\C)$ on $\C^{\Sigma_E}$ as follows.
$$\phi \mapsto \{\sigma \mapsto s(\phi(s^{-1}\sigma)\} \ \forall \ s \in
\Aut(\C).$$ The semilinearity means that
$$s(z\phi)=s(z)s(\phi) \ \forall \ z\in \C.$$
Let us consider a $\C$-linear map of $[E:\Q]$-dimensional $\C$-algebras
$$\kappa_E: E\otimes_{\Q}\C \to \C^{\Sigma_E}, \ e\otimes z \mapsto \{\sigma
\mapsto z\sigma(e)\} \ \forall \ e\in E, z\in \C$$ (see
\cite[(2.2.2)]{Deligne2}). Clearly, $\kappa_E$ is $\Aut(\C)$-equivariant. (Here
$\Aut(\C)$ acts on $E\otimes_{\Q}\C$ through its second factor in the obvious
way.) It follows from Artin's theorem on linear independence of multiplicative
characters \cite[Ch. VI, Sect. 4, Th. 4.1]{Lang} that $\kappa_E$ is injective;
now the coincidence of dimensions implies that $\kappa_E$ is an isomorphism of
$\C$-vector spaces that commutes with $\Aut(\C)$-actions. On the other hand,
for each $\sigma \in \Sigma_E$  the natural surjection
$$E\otimes_{\Q}\C \twoheadrightarrow E\otimes_{E,\sigma}\C=:\C_{\sigma}=\C$$
obviously coincides with the composition of $\kappa_E$ and
$$\C^{\Sigma_E}\to \C, \phi \mapsto \phi(\sigma).$$
This allows us to identify $\C^{\Sigma_E}$ and
$$\oplus_{\sigma\in\Sigma_E}\C_{\sigma}=\oplus_{\sigma\in\Sigma_E}\C$$ and we
may view $\kappa_E$ as an isomorphism
$$E\otimes_{\Q}\C \cong \sum_{\sigma\in \Sigma_E}\C_{\sigma}=\sum_{\sigma\in \Sigma_E}
E\otimes_{E,\sigma}\C.$$ Further we will identify $E\otimes_{\Q}\C$ with
$$\C^{\Sigma_E}=\oplus_{\sigma\in\Sigma_E}\C$$ via $\kappa_E$.
\end{sect}

\begin{sect}
\label{tracefield} Let $\G$ be a (finite) automorphism group of the field $E$
and let $F=E^{\G}$ be the subfield of $\G$-invariants. One may view $\G$ as a
certain group of automorphisms of the $\C$-algebra $E\otimes_{\Q}\C$ where $\G$
acts through the first factor. Clearly,
$$(E\otimes_{\Q}\C)^{\G}=E^{\G}\otimes_{\Q}\C=F\otimes_{\Q}\C;$$
in addition, if $$z\in E = E\otimes 1 \subset E\otimes_{\Q}\C$$ then
$$w=\Tr_{E/F}(z) \in F = F\otimes 1 \subset F\otimes_{\Q}\C$$ where
$$\Tr_{E/F}: E \to F$$
is the trace map that corresponds to the finite field extension $E/F$.

 It is
also clear that the corresponding action of $\G$ on $\C^{\Sigma_E}$ induced by
$\kappa_E$ could be described as follows. Every $s\in\G$ sends a function
$\phi:\Sigma_E \to \C$ to the function $\sigma \mapsto \phi(\sigma s)$.

If $z \in E\otimes_{\Q}\C$ then $w=\sum_{s\in \G}sz \in
(E\otimes_{\Q}\C)^{\G}=F\otimes_{\Q}\C$. If $\kappa_E(z)$ is a function $\phi$
on $\Sigma_E$ and $\kappa_F(w)$ is a function $\psi$ on $\Sigma_F$ then one may
easily check that for each field embedding $\sigma_F:F \hookrightarrow \C$
$$\psi(\sigma_F)=\sum \phi(\sigma)$$ where the sum is taken over all field
embeddings $\sigma:E \hookrightarrow \C$, whose restriction to $F$ coincides
with $\sigma_F$. In other words, if $\sigma$ is one of those embeddings then
$\psi(\sigma_F)=\sum_{s\in \G} \phi(\sigma s)$.

\end{sect}

\begin{sect}
\label{CMfields} Assume that $E$ is a CM-field and let $c_0 \in \Aut(E/\Q)$ be
the ``complex conjugation", i.e., the involution, whose subfield  of invariants
consists of all totally real elements of $E$. Since $E$ is CM, we have
$$\sigma c_0=\iota \sigma=\bar{\sigma} \ \forall \ \sigma \in \Sigma_E.$$
Let us consider the $[E:\Q]/2$-dimensional $\Q$-vector subspace
$$E_{-}:=\{e\in E \mid c_0(e)=-e\}\subset E$$
of $c_0$-antiinvariants. The involution $c_0$ gives rise to the involutions of
$\C$-algebras
$$E\otimes_{\Q}\C \to E\otimes_{\Q}\C , \ e\otimes z \mapsto c_0(e)\otimes z;$$
$$\C^{\Sigma_E} \to \C^{\Sigma_E}, \ \phi(\sigma) \mapsto \phi(\sigma
c_0)=\phi(\bar{\sigma}),$$ which we still denote by $c_0$. Clearly, $\kappa_E$
is $c_0$-equivariant. It is also clear that the $\C$-subspace of
$c_0$-antiinvariants in $E\otimes_{\Q}\C$ coincides with $E_{-}\otimes_{\Q}\C$
and the $\C$-subspace of $c_0$-antiinvariants in $\C^{\Sigma_E}$ coincides with
the subspace $X_{E,\C}$ of all functions $\phi$ that satisfy
$$\phi(\bar{\sigma})= - \phi(\sigma) \ \forall \ \sigma \in \Sigma_E.$$

Let $X_E\subset X_{E,\C}$ be the $\Q$-vector subspace  that consists of all
functions $\phi: \Sigma_E \to \Q\subset \C$ with
$$\phi(\bar{\sigma})+\phi(\sigma)=0 \ \forall \sigma\in \Sigma_E.$$
Clearly, $X_E$ is a $\Aut(\C)$-invariant $\Q$-vector subspace of
$\C^{\Sigma_E}$ and we get the natural homomorphism
$$\Aut(\C)\twoheadrightarrow \Gal(\bar{\Q}/\Q)\to \Aut_{\Q}(X_E).$$
Clearly, $\iota$ acts on $X_E$ as multiplication by $-1$.

Let
$$E^{+}=\{e\in E \mid c_0(e)=e\}$$
be the maximal totally real subfield of $E$. Clearly, $E$ is a quadratic extension of
$E^{+}$ with the Galois group $\{1, c_0\}$.  The corresponding norm map $E \to E^{+}$ coincides with the map
$$e \mapsto e \cdot c_0(e).$$
Let us extend $c_0$ by $\C$-linearity  to the $\C$-linear algebra automorphism
$$E_{\C} \to E_{\C}, \  e\otimes z \mapsto c_0(e)\otimes z,$$
which we continue to denote by  $c_0$. The corresponding automorphism of $\C^{\Sigma_E}$ (via $\kappa_E$ sends a function
$h: \Sigma_E\to \C$ to the function $\sigma \mapsto h(\sigma c_0)$.

Let $R_{E/\Q}\G_m$ and  $R_{E^{+}/\Q}\G_m$ be the algebraic $\Q$-tori obtained by the Weil restriction of scalars from the multiplicative group $\G_m$ to $\Q$ from $E$ and $E^{+}$ respectively.
For every commutative $\Q$-algebra $A$
$$R_{E/\Q}\G_m(A)=(A\otimes_{\Q}E)^*,  \  R_{E^{+}/\Q}\G_m(A)=(A\otimes_{\Q}E^{+})^{*}.$$
Clearly, $R_{E^{+}/\Q}\G_m$ is an algebraic $\Q$-subgroup of
$R_{E/\Q}\G_m$.  Again let us define the $A$-linear algebra
automorphism
$$A\otimes_{\Q}E \to A\otimes_{\Q}E, \ a\otimes e \mapsto a \otimes c_0(e),$$
which we continue denote by $c_0$. Clearly, the subalgebra of $c_0$-invariants
coincides with $A\otimes_{\Q}E^{+}$. The homomorphisms
$$(A\otimes_{\Q}E)^* \to (A\otimes_{\Q}E^{+})^{*}, \ b \mapsto b \cdot c_0(b)$$
gives rise to the $\Q$-homomorphism of algebraic  $\Q$-tori
$$R_{E/\Q}\G_m \to R_{E^{+}/\Q}\G_m,$$  whose kernel $\TT_E$ is called
the {\sl norm torus}. By definition, $$\TT_E(A)=\{b \in
(A\otimes_{\Q}E)^*\mid b \cdot c_0(b)=1 \}.$$ In particular,
$$\TT_E(\C)=\{u \in E_{\C}\mid u \cdot c_0(u)=1 \}=\kappa_E^{-1}\{h:
\Sigma_E \to \C \mid h(\sigma) h(\sigma c_0)=1 \ \forall \ \sigma\}.$$
It is well known \cite{Voskresenskii} that the norm torus is an
algebraic $\Q$-torus; in particular, it is a connected algebraic
$\Q$-group.

Let $\Q[\epsilon]=\Q\oplus \Q\cdot \epsilon$ be the $\Q$-algebra of
{\sl dual numbers}: $\epsilon^2=0$.  One may naturally identify the
$\Q$-Lie algebra $\Lie(R_{E/\Q}\G_m)$ with $E$: namely, each $e\in E$
corresponds to $1+\epsilon \otimes e \in
(\Q[\epsilon]\otimes_{\Q}E)^{*}$.  The corresponding $\Q$-Lie
subalgebras of $R_{E^{+}/\Q}\G_m$ and $\TT_E$ coincide with $E^{+}$ and
$E_{-}$ respectively.


Suppose that $E$ is a CM field that is {\sl normal} over $\Q$ and fix a field
embedding $E \hookrightarrow \bar{\Q}\subset\C$. Further, we view $E$ as a
subfield of $\C$. Then $\sigma(E)=E$ for all $\sigma$, the involution $c_0$
coincides with the restriction of the complex conjugation $\iota$ to $E$. In
addition, $c_0$ is a {\sl central} element of the Galois group $\Gal(E/\Q)$.
The set $\Sigma_E$ ``coincides" with $\Gal(E/\Q)$. In addition, the action of
$\Aut(\C)$ on $\Sigma_E=\Gal(E/\Q)$ factors through $\Gal(E/\Q)$ and
corresponds to the left translations. The action of $\Aut(\C)$ on $X_E$ factors
through $\Gal(E/\Q)$ and this action admits the following description.
$$\tau(f)(\sigma)=f(\tau^{-1}\sigma) \ \forall \tau\in \Gal(E/\Q), \sigma\in
\Sigma_E=\Gal(E/\Q), f \in X_E.$$ If we consider the $\Q$-vector (sub)space
$$E_{-}=\{e\in E\mid c_0(e)=-e\}\subset E$$
then
$$\kappa_E(E_{-}\otimes_{\Q}\C)=X_{E,\C}.$$
Clearly,
$$\dim_{\Q}(E_{-})=\frac{1}{2}[E:\Q]=\dim_{\Q}(X_E).$$
\end{sect}
\noindent \textbf{Caution:} Although
$\kappa_E(E_{-}\otimes_{\Q}\C)=X_{E,\C}$, it is \textbf{not} true that
$\kappa_E(E_{-})= X_E$ unless both are 0. Indeed, for any nonzero $e\in E_{-}$, the function
$\kappa_E(e)$ takes value $\sigma(e)$ at $\sigma\in \Sigma_E$,
which is never real, while $X_E$ consists of $\Q$-valued functions.

\begin{rem}
The $\Gal(E/\Q)$-module $X_E$ is faithful. Indeed, let us consider the function
$f$ on $\Sigma_E=\Gal(E/\Q)$ that takes on value $1$ on the identity element of
$\Gal(E/\Q)$, value $-1$ on $c_0$ and zero elsewhere. Then $f\in X_E$ but
$\tau(f) \ne f$ if $\tau$ is {\sl not} the identity element of $\Gal(E/\Q)$.
\end{rem}

\begin{defn}
We write $\max(E)$ for the largest $\Q$-dimension of simple
$\Gal(E/\Q)$-submodules of $X_E$. Clearly, $\max(E)\le \dim_{\Q}(X_E)$; the
equality holds if and only if $X_E$ is simple.
\end{defn}

\begin{lem}\label{lem:simple-module}
  Let $G=\Gal(E/\Q)$ and $W$ be a simple $\Q[G]$-module such that the
  involution $c_0$ acts on $W$ as multiplication by $-1$. Then there exists an
  injective homomorphism of $\Q[G]$-modules $W\hookrightarrow X_E$. In
  particular, $X_E$ contains a $\Q[G]$-submodule that is isomorphic to
  $W$.
\end{lem}
\begin{proof}
  Fix a nonzero linear function $\lambda\in \Hom_{\Q}(W,\Q)$ and
  consider the $\Q$-linear map
\[ \pi_{\lambda}: W\to \Q^{\Sigma_E}, \quad x \mapsto \{\sigma\mapsto
\lambda(\sigma^{-1}x)\}.\]
Clearly $\pi_{\lambda}$ is nonzero.
For all $x\in W$, $\sigma\in \Sigma_E$ and $\tau \in G$,  we have
\[\pi_{\lambda}(\tau x)(\sigma)=\lambda(\sigma^{-1}\tau
x)=\lambda((\tau^{-1}\sigma)^{-1}x)=\pi_{\lambda}(x)(\tau^{-1}\sigma)=\tau(\pi_{\lambda}(x))(\sigma).\]
Hence $\pi_{\lambda}$ is a map of $\Q[G]$-modules. In particular, if
we choose $\tau$ to be the involution $c_0$, then
\[ c_0(\pi_{\lambda}(x))=\pi_{\lambda}(c_0 x)=\pi_{\lambda}(-x)=-\pi_{\lambda}(x).\]
It follows that $\pi_{\lambda}(W)\subseteq X_E$ and we have a map of
$\Q[G]$-modules $W\to X_E\subset \Q^{\Sigma_E}$ that is still denoted
by $\pi_{\lambda}$. Now the lemma follows since $W$ is simple and
$\pi_{\lambda}$ is nonzero.
\end{proof}

\begin{exs}
\label{maxcycl} (i) Suppose that $\Gal(E/\Q)=\langle c_0\rangle \times H$
where  $H$ is a cyclic subgroup of order $M$ in $G$. Let us consider the
$G$-module $\Q(\zeta_{M})$ where the group $H\cong \mu_{M}$ acts via
multiplication by $M$th roots of unity and $c_0$ acts as multiplication by
$-1$. Clearly, $\Q(\zeta_{M})$ is simple and
$\dim_{\Q}(\Q(\zeta_{M}))=\varphi(M)$. It follows from Lemma
\ref{lem:simple-module} that the $G$-module $X_E$ contains a submodule that is
isomorphic to $\Q(\zeta_{M})$. In particular, $\max(E) \ge \varphi(M)$. (In
fact, one may prove that $\max(E) = \varphi(M)$.)

\smallskip

(ii) Suppose that $G$ is a cyclic group of order $2M$. Then $c_0$ is its only element
of order $2$. Let us consider the $G$-module $\Q(\zeta_{2M})$ where the group
$G\cong \mu_{2M}$ acts via multiplication by $2M$th roots of unity. Clearly,
$c$ acts on $\Q(\zeta_{2M})$ as multiplication by $-1$. It is also clear that
the $G$-module $\Q(\zeta_{2M})$ is simple. It follows again that the $G$-module $X_E$
contains a submodule that is isomorphic to $\Q(\zeta_{2M})$. In particular,
$\max(E) \ge \varphi(2M)$. (In fact, one may prove that $\max(E) =
\varphi(2M)$.)
\end{exs}


\begin{lem}
\label{CMdim}
 Let $E$ be a CM-field that is normal over $\Q$ and let us fix an
embedding $E \hookrightarrow \C$. (Further we we view $E$ as a subfield of
$\C$.) Let $h:\Sigma_E\to \Q\subset\C$ be a $\Q$-valued function on $\Sigma_E$
that lies in $X_E$. Let $W$ be the $\Q$-vector subspace of $X_E$ generated by
all $\tau(h):\sigma\mapsto h(\tau^{-1}\sigma)$ where $\tau$ runs through $\Gal(E/\Q)$. Let $\q$ be the smallest
$\Q$-vector (sub)space of $E_{-}$ such that $\kappa_E(\q_{\C})$ contains $h$.
Then
$$\dim_{\Q}(\q)=\dim_{\Q}(W).$$
In particular, $\q=E_{-}$ if and only if $W=X_E$.
\end{lem}

\begin{proof}
By definition, $W\subset X_E$ is the $\Q$-vector subspace generated by
functions $h(s^{-1}\sigma)$, $s \in \Gal(E/\Q)$. Clearly, $\dim_{\Q}(W)$
coincides with the rank of the matrix $(a_{s,\sigma})=(h(s^{-1}\sigma))$  over
the rationals with $s \in \Gal(E/\Q), \sigma \in \Sigma_E$. Let $\tilde{W}
\subset X_{E,\C}$ be the $\C$-vector (sub)space generated by functions
$h(s^{-1}\sigma)$, $s \in \Gal(E/\Q)$.  Clearly, $\dim_{\C}(\tilde{W})$
coincides with the rank of the matrix $(a_{s,\sigma})=(h(s^{-1}\sigma))$ over
the complex numbers with $s \in \Gal(E/\Q), \sigma \in \Sigma_E$. In
particular,
$$\dim_{\Q}(W)=\dim_{\C}(\tilde{W}).$$
It is also clear that $\tilde{W}$ is the smallest $\Aut(\C)$-invariant complex
vector subspace of $\C^{\Sigma_E}$ that contains $h(\sigma)$. It follows that there
exists a $\Q$-vector subspace $\q^{\prime}\subset E$ such that
$\q^{\prime}_{\C}=\q^{\prime}\otimes_{\Q}\C$ coincides with $\tilde{W}$. In
particular
$$\dim_{\Q}(\q^{\prime})=\dim_{\C}(\tilde{W}).$$
The minimality property of $\q$ implies that $\q\subset \q^{\prime}$ and
therefore
$$h \in \q_{\C}\subset \q^{\prime}_{\C}=\tilde{W}.$$
The minimality property of $\tilde{W}$ implies that $\q_{\C}=\tilde{W}$ and
therefore $\q_{\C}= \q^{\prime}_{\C}$. Since $\q\subset \q^{\prime}$, we
conclude that $\q=\q^{\prime}$. In order to finish the proof, one has only to
recall that
$$\dim_{\Q}(\q)=\dim_{\C}(\tilde{W})=\dim_{\Q}(W).$$
\end{proof}

\begin{sect}
\label{fieldproducts} Let $t$ be a positive integer and suppose that for each
positive $j\le t$ we are given a number field $E_j$. For the sake of
simplicity, let us assume that every $E_j$ is {\sl normal} over $\Q$ and write
$\Gal(E_j/\Q)$ for the corresponding Galois group. Further, we fix an embedding
of $E_j$ into $\C$; this allows us to identify $\Sigma_{E_j}$ and
$\Gal(E_j/\Q)$. Let us consider the product $$\E =\prod_{j=1}^t
E_j=\oplus_{j=1}^t E_j.$$ Clearly, $\E$ is is a finite-dimensional semisimple
commutative $\Q$-algebra and the set $\Sigma_{\E}$ of algebra homomorphisms $\E
\to \C$ that send $1$ to $1$ could be naturally identified with the disjoint
union $\coprod _{j=1}^t \Sigma_{E_j}$ of $\Sigma_{E_j}$'s. Taking the product
of $\kappa_{E_j}$'s, we get the natural isomorphism of $\C$-algebras
$$\kappa_{\E}:\E_{\C} \cong \C^{\Sigma_{\E}},$$
which sends $\{e_j\}_{j=1}^t\otimes z$ to the function
$$\Sigma_{\E}=\coprod _{j=1}^t \Sigma_{E_j} \to \C$$
that coincides with $\sigma \mapsto \sigma(e_j) z$ on $\Sigma_{E_j}$. As above,
we identify $\E_{\C}$ with  the space of functions $\C^{\Sigma_{\E}}$ via
$\kappa_{\E}$. Again, there is the natural semilinear action of $\Aut(\C)$ on
$\E_{\C}$, whose subalgebra of invariants coincides with $\E\otimes 1=\E$. An
automorphism $\tau\in \Aut(\C)$ sends function $h:\Sigma_{\E}\to \C$ to the
function $\tau(h):=\{\sigma \to \tau(h(\tau^{-1}\sigma))\}$. We have a
$\Aut(\C)$-invariant splitting
$$\C^{\Sigma_{\E}}=\oplus_{j=1}^t \C^{\Sigma_{E_j}}.$$
Clearly, every function $h:\Sigma_{\E} \to \C$ may be viewed as a collection
$\{h_j\}_{j=1}^t$ of functions $h_j: \Sigma_{E_j}\to \C$. The $\Q$-vector
(sub)space $\Q^{\Sigma_{\E}}$ of $\Q$-valued functions is $\Aut(\C)$-invariant;
in addition, the action of $\Aut(\C)$ on $\Q^{\Sigma_{\E}}$ factors through
$\Gal(\bar{\Q}/\Q)$ and for all $\Q$-valued functions $h$
$$\tau(h)(\sigma)=h(\tau^{-1}\sigma) \ \forall \sigma\in\Sigma_{\E}, \tau \in
\Gal(\bar{\Q}/\Q).$$ We have a $\Gal(\bar{\Q}/\Q)$-invariant splitting
$$\Q^{\Sigma_{\E}}=\oplus_{j=1}^t \Q^{\Sigma_{E_j}}.$$
\end{sect}

\begin{lem}
\label{periodDIM}
 Let $h=\{h_j\}_{j=1}^t$ be a function on $\Sigma_{\E}$
that takes on only rational values, i.e., $h_j(\Sigma_{E_j})\subset \Q \
\forall j$. Let $W$ (resp. $W_j$) be the $\Q$-vector subspace generated by all
$\tau(h)$ (resp. $\tau(h_j)$) where $\tau$ runs through $\Gal(\bar{\Q}/\Q)$. We
have
$$W\subset \Q^{\Sigma_{\E}}, \  W_j\subset \Q^{\Sigma_{E_j}}, \ W \subset
\oplus_{j=1}^t W_j.$$ On the other hand, let $\q$ (resp. $\q_j$) be the
smallest $\Q$-vector subspace of $\E$ (resp. of $E_j$ such that
$\kappa_{\E}(\q_{\C})$ contains $h$ (resp. $\kappa_{E_j}(\q_j\otimes_{\Q}\C)$
contains $h_j$). Then
$$\dim_{\Q}(W)=\dim_{\Q}(\q); \ \dim_{\Q}(W_j)=\dim_{\Q}(\q_j) \ \forall j.$$
\end{lem}

\begin{proof}
The proof could be carried out by the same arguments as the proof of Lemma
\ref{CMdim} and is left to the reader.
\end{proof}

If $F$ is a subfield of $E_t$ then we write $\Tr_{E_t/F}: E_t \to F$ for the
corresponding $\Q$-linear trace map.  Extending $\Tr_{E_t/F}$ by
$\C$-linearity, we get a $\C$-linear map
$$E_t \otimes_{\Q}\C \to F\otimes_{\Q}\C,$$
which we still denote by $\Tr_{E_t/F}$.

\begin{lem}
  \label{tracehodge}Assume that for all $j$ the field $E_t$ contains
  $E_j$. Let
$$x=\{x_j\}_{j=1}^t \in \prod _{j=1}^t (E_j)_{\C}=\E_{\C}.$$
Suppose that for all $j$
$$x_j=\Tr_{E_t/E_j}(x_t).$$
Let $\q$ (resp. $\q_t$) be the smallest $\Q$-vector subspace of $\E$ such that
$\q_{\C}$ contains $x$ (resp. the smallest $\Q$-vector subspace of $E_t$ such
that $(\q_t)_{\C}$ contains $x_t$). Then
$$\q=\{ (e_j)_{j=1}^t \in \prod _{j=1}^t E_j=\E\mid e_t\in \q_t, e_j=\Tr_{E_t/E_j}(e_t) \ \forall
j\}.$$ In particular, $\dim_{\Q}(\q)=\dim_{\Q}(\q_t)$.
\end{lem}

\begin{proof}
Let us put
$$\q^{\prime}=\{ (e_j)_{j=1}^t \in \E\mid e_t\in \q_t, e_j=\Tr_{E_t/E_j}(e_t) \ \forall
j.\}$$  Clearly, $\dim_{\Q}(\q^{\prime})=\dim_{\Q}(\q_t)$ and
$${\q^{\prime}}_{\C}=\{(z_j)_{j=1}^t \in \prod _{j=1}^t (E_j)_{\C}=\E_{\C}\mid z_t\in (\q_t)_{\C}, z_j=\Tr_{E_t/E_j}(z_t) \ \forall
j\}.$$ It is also clear that ${\q^{\prime}}_{\C}$ contains $x$. The minimality
property of $\q$ implies that $\q\subset \q^{\prime}$ and therefore
$$\dim_{\Q}(\q)\le \dim_{\Q}(\q^{\prime})=\dim_{\Q}(\q_t);$$
the equality holds if and only if $\q=\q^{\prime}$. On the other hand, let us
consider the projection map $\q \subset \E=\prod _{j=1}^t E_j
\twoheadrightarrow E_t$. The minimality properties for $\q$ and $\q_t$ imply
that the image of $\q$ coincides with $\q_t$; in particular,
$$\dim_{\Q}(\q)\ge\dim_{\Q}(\q_t).$$
This proves that
$$\dim_{\Q}(\q)=\dim_{\Q}(\q_t)=\dim_{\Q}(\q^{\prime})$$
and therefore $\q=\q^{\prime}$.
\end{proof}

\begin{thm}
\label{TRACEHG} Keep the notation and assumptions of Lemma \ref{periodDIM}.
Assume additionally that
for all $j$ the field $E_t$ contains $E_j$ and for each $\sigma_j \in
\Sigma_{E_j}$ we have $h_j(\sigma_j)=\sum_{\sigma} h_t(\sigma)$ where the sum
is taken across all $\sigma : E_t\hookrightarrow \C$, whose restriction to $E_j$
coincides with $\sigma_j$. Then
$$\q=\{ (e_j)_{j=1}^t \in \E\mid e_t\in \q_t, e_j=\Tr_{E_t/E_j}(e_t) \ \forall
j\}.$$ In particular, $\dim_{\Q}(\q)=\dim_{\Q}(\q_t)$.
\end{thm}

\begin{proof}
 Since $E_j$ is normal over $\Q$, field extension $E_t/E_j$ is also
normal and we may view the Galois group $\G_j:=\Gal(E_t/E_j)$ as a normal
subgroup of $\G:=\Gal(E_t/\Q)$.

Recall that group $\G$ acts naturally by $\C$-linear automorphism on
$(E_t)_{\C}$ by
$$s(e\otimes z)=s(e)\otimes z \ \forall s\in\G, \ e\in E_t, z\in \C$$
and on $\C^{\Sigma_{E_t}}$ via
$$(su)(\sigma)=u(\sigma s) \ \forall s\in\G, \sigma \in \Sigma_{E_t}, u \in
\C^{\Sigma_{E_t}}.$$ Clearly, the isomorphism $\kappa_{E_t}$ is
$\G$-equivariant.

It is also clear that if $\sigma_j: E_j \hookrightarrow \C$ is a field
embedding and $\sigma_t: E_t \hookrightarrow \C$ is a field embedding that
extends $\sigma_j$ then the coset $\sigma_t\G_j$ coincides with the set of all
field embeddings $\sigma:E_t \hookrightarrow \C$, whose restriction to $E_j$
coincides with $\sigma_j$. It follows that
$$h_j(\sigma_j)=\sum_{s \in \G_j}h_t(\sigma s).$$
This implies that if we put $x=\kappa_{\E}^{-1}(h)$ then
$$x=\{x_j\}_{j=1}^t \in \prod _{j=1}^t (E_j)_{\C}=\E_{\C}$$
satisfies
$$x_j=\Tr_{E_t/E_j}(x_t) \ \forall j.$$
Now the result follows from Lemma \ref{tracehodge}.
\end{proof}

\section{Complex abelian varieties}
\label{MT}

\begin{sect}
Let $Z$ be a complex abelian variety of positive dimension. We write $\CC_Z$
for the center of the semisimple finite-dimensional $\Q$-algebra $\End^0(Z)$.
Let us choose a polarization on $Z$ and let
$$\End^0(Z) \to \End^0(Z), \ u \mapsto u^{\prime}$$
be  the corresponding Rosati involution. It is well-known that $\CC_Z$ is
stable under the Rosati involution and its restriction
$$\CC_Z \to \CC_Z, \ u \mapsto u^{\prime}$$
does not depend on the choice of polarization. In addition, if $\CC_Z$ is a
CM-field $E$ then the Rosati involution on $E$ coincides with the complex
conjugation $c_0$.
\end{sect}

\begin{sect}
 Let $\H_1(Z,\Q)$ be the
first rational homology group of $Z$: it is a $2\dim(Z)$-dimensional
$\Q$-vector space. The $\Q$-algebra $\End^0(Z)$ acts by functoriality on
$\H_1(Z,\Q)$ and this action gives rise to the embedding of $\Q$-algebras
$$\End^0(Z)\hookrightarrow \End_{\Q}(\H_1(Z,\Q))$$
that sends the identity automorphism $1_Z$ of $Z$ to the identity automorphism
$\I$ of $\H_1(Z,\Q)$. It follows easily \cite[Ch. II]{Shimura} that if $E
\subset \End^0(Z)$ is a subfield that contains $1_Z$ then $E$ is a number field
and the embedding
$$E \subset \End^0(Z) \hookrightarrow \End_{\Q}(\H_1(Z,\Q))$$ provides $\H_1(Z,\Q)$ with
the natural structure of an $E$-vector space of dimension
$$d=d(Z,E):=\frac{2\dim(Z)}{[E:\Q]}.$$
We write $\End_E(\H_1(Z,\Q))\subset \End_{\Q}(\H_1(Z,\Q))$ for the $E$-algebra of
$E$-linear operators in $\H_1(Z,\Q)$ and
$$\Tr_E: \End_E(\H_1(Z,\Q)) \to E$$ for the corresponding trace
map. Clearly, $\Tr_E$ is a $\Q$-Lie algebra homomorphism (even an
$E$-Lie algebra homomorphism). Here $E$ is viewed as a commutative Lie
algebra.

Let us consider the first complex homology group of $Z$
$$\H_1(Z,\C)=\H_1(Z,\Q)\otimes_{\Q}\C,$$
which is a $2\dim(Z)$-dimensional complex vector space. If $E$ is as
above then $\H_1(Z,\C)$ carries the natural structure of a free
$E_{\C}:=E\otimes_{\Q}\C$-module of rank $d(Z,E)$. We write
$\End_{E_{\C}}(\H_1(Z,\C))\subset \End_{\C}(\H_1(Z,\C))$ for
$E_\C$-algebra of endomorphisms of the free $E_{\C}$-module
$\H_1(Z,\C)$ and
$$\Tr_{E_{\C}}: \End_{E_{\C}}(\H_1(Z,\C)) \to E_{\C}$$ for the corresponding trace
map. For example,
$$\Tr_{E_{\C}}(\I_{\C})=d(Z,E)=d.$$
Here  $\I_{\C}$ stands for the identity automorphism of $\H_1(Z,\C)$.

The group $\Aut(\C)$  acts {\sl tautologically} on
$\H_1(Z,\C)=H_1(Z,\;\Q)\otimes_{\Q}\C$ by semilinear automorphisms through the
second factor. The natural homomorphism of $\C$-algebras
$$\End_{\Q}(\H_1(Z,\Q))\otimes_{\Q}\C \to
\End_{\C}(\H_1(Z,\Q)\otimes_{\Q}\C)=\End_{\C}(\H_1(Z,\C))$$ is an isomorphism
that will allow us to identify $\C$-algebras
$\End_{\Q}(\H_1(Z,\Q)\otimes_{\Q}\C$ and $\End_{\C}(\H_1(Z,\C))$. The group
$\Aut(\C)$ acts tautologically on
$\End_{\C}(\H_1(Z,\C))=\End_{\Q}(\H_1(Z,\Q))\otimes\C $ by semilinear
automorphisms.
\end{sect}

\begin{sect}
 There is a
canonical Hodge decomposition (\cite[chapter 1]{Mumford},
\cite[pp.~52--53]{Deligne})
$$\H_1(Z,\C)=\H^{-1,0} \oplus \H^{0,-1}$$
where $\H^{-1,0}=\H^{-1,0}(Z)$ and  $\H^{0,-1}=\H^{0,-1}(Z)$ are mutually ``complex
conjugate" $\dim(Z)$-dimensional complex vector spaces. This splitting is
$\End^0(Z)$-invariant (and the $\End^0(Z)$-module $\H^{-1,0}$ is canonically
isomorphic to the commutative Lie algebra $\Lie(Z)$ of $Z$). Let
$$\f_H=\f_{H,Z}:\H_1(Z,\C) \to
\H_1(Z,\C)$$ be the $\C$-linear operator in $\H_1(Z,\C)$ defined as follows.
$$\f_H(x) =-x \quad \forall \ x \in \H^{-1,0}; \quad \f_H(x)=0 \quad
\forall \ x \in \H^{0,-1}.$$ Clearly, $\f_H$ commutes with $\End^0(Z)$.

Suppose that $\MT=\MT_Z \subset \GL_{\Q}(\H_1(Z,\Q))$ is the Mumford-Tate group
of (the rational Hodge structure $\H_1(Z,\Q)$ and of) $Z$
(\cite{Deligne,Ribet3,ZarhinIzv}). It is a connected reductive algebraic
$\Q$-group that contains scalars and could be described as follows
(\cite[section 6.3]{ZarhinIzv}). Let $\mt\subset \End_{\Q}(\H_1(Z,\Q))$ be the
$\Q$-Lie algebra of $\MT$; it is a reductive algebraic linear $\Q$-Lie algebra
which contains scalars and  its natural faithful representation in $\H_1(Z,\Q)$
is completely reducible. In addition, $\mt$ is the {\sl smallest} $\Q$-Lie
subalgebra in $\End_{\Q}(\H_1(Z,\Q))$ that enjoys the following property: its
complexification
$$\mt_{\C}=\mt\otimes_{\Q}\C \subset \End_{\C}(\H_1(Z,\C))$$
contains scalars and $\f_H$. It is well-known that the centralizer of $\MT$
(and therefore of $\mt$) in $\End_{\Q}(\H_1(Z,\Q))$ coincides with $\End^0(Z)$.
This implies that the center $\cc$ of $\mt$ lies in $\CC_Z$. Since $\mt$ is
reductive, it splits into a direct sum
$$\mt=\mt^{ss} \oplus \cc$$
of $\cc$ and a semisimple $\Q$-Lie algebra $\mt^{ss}$.

Since $\mt^{ss}$ is semisimple, and $E$ is commutative,
$$\Tr_{E}(\mt)=\Tr_{E}(\cc)\subset E.$$
This implies easily that
$$\Tr_{E_{\C}}(\mt_{\C}) =\Tr_{E}(\cc)\otimes_{\Q}\C\subset E\otimes_{\Q}\C.$$
In particular, since $\f_H \in \mt_{\C}$, we have $\Tr_{E_{\C}}(\f_H) \in
\Tr_{E}(\cc)\otimes_{\Q}\C$.

\end{sect}

\begin{sect}
\label{hodge}

 We refer to \cite{Ribet3}, \cite[Sect. 6.6.1 and 6.6.2]{ZarhinIzv}
for the definition and basic properties of the Hodge group $\Hdg=\Hdg_Z$ of the
rational Hodge structure $\H_1(Z,\Q)$ and of $Z$. Recall that $\Hdg$ is a
normal connected algebraic subgroup of $\MT$; in addition, $\Hdg$ lies in the
special general linear group $\SL_{\Q}(\H_1(Z,\Q))$ of $\H_1(Z,\Q)$ and the
natural homomorphism-product
$$\Hdg \times \G_m \to \MT$$
is an isogeny of connected algebraic $\Q$-groups. Here $\G_m
=\G_m\cdot\I\subset \GL_{\Q}(\H_1(Z,\Q))$ is the group of homotheties. It
follows easily that $\Hdg$ is reductive and if $$\hdg=\hdg_Z\subset
\End_{\Q}(\H_1(Z,\Q))$$ is the $\Q$-Lie algebra of $\Hdg$ then it is reductive,
its semisimple part coincides with $\mt^{ss}$ and
$$\mt=\Q\cdot\I\oplus \hdg, \ \hdg=\mt \bigcap \sL(\H_1(Z,\Q)).$$
(Here $\sL(\H_1(Z,\Q))$ is the (simple) $\Q$-Lie algebra of $\Q$-linear
operators in $\H_1(Z,\Q)$ with zero trace.)
 In particular, if $\cc^{0}=\cc^{0}_Z$ is the center of (reductive) $\hdg$ then
$$\cc=\cc^{0}\oplus\Q\cdot\I, \ \hdg=\mt^{ss}\oplus \cc^{0}, \ \mt=\mt^{ss}\oplus
\cc^{0}\oplus\Q\cdot\I.$$ Clearly,
$$\mt_{\C}=\hdg_{\C}\oplus \C\cdot \I_{\C}, \ \f_H =\left(\f_H
+\frac{1}{2}\I_{\C}\right)-\frac{1}{2}\I_{\C}=\f_H^{0}-\frac{1}{2}\I_{\C}$$
where
$$\f_H^{0}:=\f_H
+\frac{1}{2}\I_{\C} \in \sL(\H_1(Z,\C)).$$
It follows easily that $\hdg$ is the {\sl smallest} $\Q$-Lie subalgebra in
$\End_{\Q}(\H_1(Z,\Q))$ that enjoys the following property: its
complexification
$$\hdg_{\C}=\hdg\otimes_{\Q}\C \subset \End_{\C}(\H_1(Z,\C))$$
contains $\f_H^{0}$. Clearly,
\begin{equation}
  \label{eq:trhdg}
  \Tr_{E}(\hdg)=\Tr_{E}(\mt^{ss}\oplus \cc^{0})=\Tr_{E}(\cc^{0}).
\end{equation}

The choice of the polarization on $Z$ gives rise to an alternating
non-degenerate $\Q$-bilinear form
$$\psi_{\Q}: \H_1(Z,\Q) \times \H_1(Z,\Q) \to \Q$$
that is $\Hdg$-invariant; in addition
$$\psi_{\Q}(ux,y)=\psi_{\Q}(x,u^{\prime}y) \ \qquad \forall u\in \End^0(Z), \ x,y \in
\H_1(Z,\Q).$$ The $\Hdg$-invariance of $\psi_{\Q}$ means that
$$\psi_{\Q}(ux,y)+\psi_{\Q}(x,uy)=0 \ \qquad \forall u\in \hdg, \ x,y \in
\H_1(Z,\Q).$$ If $u \in \cc^{0} \subset \hdg$ then $u \in \CC_Z$ and we have
$$\psi_{\Q}(ux,y)=\psi_{\Q}(x,u^{\prime}y), \
\psi_{\Q}(ux,y)+\psi_{\Q}(x,uy)=0.$$ Since $(u^{\prime})^{\prime}=u$, we have
$\psi_{\Q}(u^{\prime}x,y)=\psi_{\Q}(x,uy)$ and therefore
$$0=\psi_{\Q}(ux,y)+\psi_{\Q}(x,uy)=\psi_{\Q}(ux,y)+\psi_{\Q}(u^{\prime}x,y)=\psi_{\Q}((u+u^{\prime})x,y).$$
The non-degeneracy of $\psi_{\Q}$ implies that $u+u^{\prime}=0$, i.e.,
$u^{\prime}=-u$. This means that
$$\cc^{0} \subset \{u \in \CC_Z\mid u^{\prime}=-u\}\subset \CC_Z.$$

\begin{rem}
\label{centerhdg}
It is well known \cite{Mumford} that if the center $\CC_Z$ is a field then it
is either a totally real number field or a CM-field. If $\CC_Z$ is a totally
real number field then the Rosati involution acts on $\CC_Z$ as identity map,
$\{u \in \CC_Z\mid u^{\prime}=-u\}=\{0\}$ and therefore $$\cc^{0}=\{0\}.$$

Suppose that $\CC_Z$ is a CM field, i.e., a totally imaginary quadratic
extension of a totally real number field $F_Z$. Then the Rosati involution acts
on  on $\CC_Z$ as the ``complex conjugation" \cite{Mumford}; in particular, it
is $F_Z$-linear and
 $\{u \in \CC_Z\mid u^{\prime}=-u\}$ is a one-dimensional $F_Z$-vector
 subspace of $\CC_Z$ and therefore its $\Q$-dimension equals $[\CC_Z:\Q]/2$.
 This implies that
$$\cc^{0} \subset \{u \in \CC_Z\mid u^{\prime}=-u\}, \ \dim_{\Q}(\cc^{0}) \le
\frac{1}{2}[\CC_Z:\Q].$$

If ${Z}=\prod_{j=1}^t Z_j$ is a product of abelian varieties $Z_j$'s  then
there is an inclusion $\oplus_{j=1}^t \End^0(Z_j)\subset \End^0(Z)$  and
therefore $\CC_Z \subset \oplus_{j=1}^t \CC_{Z_j}$.
\end{rem}

\end{sect}

\begin{sect}
\label{HdgCenter}
Suppose that a CM field $E$ is the center of $\End^0(Z)$.  As  in Subsect. \ref{CMfields} ,  we write $c_0$ for the ``complex conjugation" on $E$ and $\TT_E$ for the corresponding norm torus.
 Clearly, the center of  the $\C$-algebra
$$\End^0(Z)_{\C}=\End^0(Z)\otimes_{\Q}\C\subset   \End_{\Q}(\H_1(Z,\Q))\otimes_{\Q}\C= \End_{\C}(\H_1(Z,\C))$$
 coincides with $E_{\C}$.

 Let $\ZZ$ be the center of $\Hdg$.

The inclusion $E \subset \End^0(Z)\subset \End_{\Q}(\H_1(Z,\Q))$
gives rise to the embedding of $\Q$-algebraic groups $R_{E/\Q}\G_m
\subset \GL(\H_1(Z,\Q))$. Since $\TT_E \subset R_{E/\Q}\G_m$, we have
$$\TT_E \subset R_{E/\Q}\G_m \subset \GL(\H_1(Z,\Q)), $$
$$ R_{E/\Q}\G_m(\Q)=E^{*} \subset  \Aut_{\Q}(\H_1(Z,\Q)), \
\TT_E(\Q)=\{e\in E\mid e\, c_0(e)=1\}. $$ Clearly, the $\Q$-Lie algebras
of $\TT_E$ and $R_{E/\Q}\G_m$, viewed as $\Q$-Lie subalgebras of
$ \End_{\Q}(\H_1(Z,\Q))$, coincide with $E_{-}$ and $E$ respectively.
Since $\H_1(Z,\C)=\H_1(Z,\Q)\otimes_{\Q}\C,$ we have
$$R_{E/\Q}\G_m(\C)=E_{\C}^* \subset \Aut_{\C}(\H_1(Z,\C)), $$
$$  \TT_E(\C)=\{u \in E_{\C}^{*}\mid u \cdot c_0(u)=1\} \subset E_{\C}^* \subset \Aut_{\C}(\H_1(Z,\C)) .$$

Since the centralizer of $\hdg$ in $\End_{\Q}(\H_1(Z,\Q))$ coincides with
$\End^0(Z)$ , it follows that the centralizer of the $\C$-Lie algebra
$\hdg_{\C}$ in $\End_{\C}(\H_1(Z,\C))$ coincides with $\End^0(Z)_{\C}$.  Since
the $\C$-Lie subalgebra
$$\hdg_{\C}\subset\End_{\C}(\H_1(Z,\C))$$
coincides with the $\C$-Lie algebra of the \textit{connected} complex
algebraic subgroup $\Hdg(\C)\subset \Aut_{\C}(\H_1(Z,\C))$, it follows
that the centralizer of $\Hdg(\C)$ in $\End_{\C}(\H_1(Z,\C))$ also
coincides with $\End^0(Z)_{\C}$.  This implies that the center
$\ZZ(\C)$ of $\Hdg(\C)$ lies in the center of $\End^0(Z)_{\C}$.  It
follows that
$$\ZZ(\C)\subset E_{\C}^*=R_{E/\Q}\G_m(\C).$$
This implies that
$$\ZZ\subset R_{E/\Q}\G_m.$$
 We want to prove that $\ZZ \subset \TT_E$.
In order to do that, let us extend  $\psi_{\Q}$   by $\C$-linearity to $\H_1(Z,\Q)\otimes_{\Q}\C=\H_1(Z,\C)$. We get a non-degenerate alternating $\C$-bilinear form
$$\psi_{\C}: \H_1(Z,\C) \times \H_1(Z,\C) \to \C,$$
which is  $\Hdg(\C)$-invariant.  Clearly,
$$\psi_{\C}(ux,y)=\psi_{\C}(x, c_0(u)y) \ \forall u \in E_{\C}, \ x,y \in \H_1(Z,\C).$$
This implies that
$$\psi_{\C}(ux,uy)=\psi_{\C}(x, c_0(u)u y)= \psi_{\C}(x, u c_0(u) y), \ \forall u \in E_{\C}.$$
This implies that if $u \in E_{\C}$ then $\psi_{\C}$ is $u$-invariant if and only if  $u c_0(u)=1$, i.e.,
$u \in \TT_E(\C)$. It follows that
$\ZZ(\C)\subset \TT_E(\C)$, i.e.,
$$\ZZ \subset \TT_E.$$
\end{sect}

\begin{sect}
The $\dim(Z)$-dimensional complex vector space $\Omega^1(Z)$ of the
differentials of the first kind on $Z$ carries the natural structure of
$E\otimes_{\Q}\C$-module \cite[Sect. 2]{ZarhinCamb}. Clearly,
$$\Omega^1(Z)=\bigoplus_{\sigma\in
\Sigma_E}\C_{\sigma}\Omega^1(Z)=\oplus_{\sigma\in
\Sigma_E}\Omega^1(Z)_{\sigma}$$ where
$\Omega^1(Z)_{\sigma}:=\C_{\sigma}\Omega^1(Z)=\{x \in \Omega^1(Z)\mid
ex=\sigma(e)x \quad \forall e\in E\}$. Let us put
$$n_{\sigma}=n_{\sigma}(Z,E)=\dim_{\C_{\sigma}}\Omega^1(Z)_{\sigma}=\dim_{\C}\Omega^1(Z)_{\sigma}.$$
It follows (compare with  \cite[p. 260]{ZarhinCamb}) that
$\Tr_{E_{\C}}(\f_H)=(-n_{\sigma})_{\sigma \in\Sigma_E}$. This implies  that
$$(n_{\sigma})_{\sigma \in\Sigma_E}=-\Tr_{E_{\C}}(\f_H)\in
\Tr_{E}(\cc)\otimes_{\Q}\C=\Tr_{E}(\mt)\otimes_{\Q}\C.$$
\end{sect}

\begin{rems}
\label{function}

\begin{itemize}

\item[(i)]

It is well-known \cite[Sect. 2]{ZarhinCamb} that
$$n_{\sigma}+n_{\bar{\sigma}}=d=2\dim(Z)/[E:\Q] \ \forall \ \sigma .$$
This means that the function
$$\Sigma_E \to \Q, \ \sigma \mapsto \frac{d}{2}-n_{\sigma}$$
lies in $X_E$.

\item[(ii)] Recall  that the Hodge splitting commutes with $\End^0(Z)$ and
therefore with $E$. Hence $\f_H$ may be viewed as an endomorphism of the free
$E_{\C}$-module $\H_1(Z,\C)$ and its trace in $E_{\C}$ is the tuple
$$(-n_{\sigma})_{\sigma
\in\Sigma_E}\in\prod_{\sigma\in\Sigma_E}\C_{\sigma}=E_{\C}$$  \cite[Sect.
2]{ZarhinCamb}. It follows that
$$\Tr_{E_{\C}}(\f_H^{0})=\Tr_{E_{\C}}(\f_H)+\Tr_{E_{\C}}\left(\frac{1}{2}\I_{\C}\right)=
\Tr_{E_{\C}}(\f_H)+\frac{1}{2}d=\left\{\frac{d}{2}-n_{\sigma}\right\}_{\sigma\in
\Sigma_E}\in E_{\C}.$$

\end{itemize}

\end{rems}

\begin{lem}
\label{traceminimal} $\Tr_{E}(\hdg)$ coincides with the smallest $\Q$-vector
subspace $\q \subset E$ such that the $\C$-vector subspace
$$\q_{\C}=\q\otimes_{\Q}\C \subset E_{\C}=E\otimes_{\Q}\C=\oplus_{\sigma\in \Sigma_E}\C_{\sigma}=\C^{\Sigma_E}$$
contains $\{\frac{d}{2}-n_{\sigma}\}_{\sigma \in\Sigma_E}$.
\end{lem}

\begin{proof}
Clearly, $\Tr_{E}(\hdg)$ contains $\q$, because $\hdg_{\C}$ contains
$\f_H^{0}$. On the other hand, if $\Tr_{E}(\hdg)\ne \q$ then
$$\hdg^{\prime}: =\{u\in \hdg\mid \Tr_{E}(u)\in \q\}$$
is a {\sl proper} $\Q$-Lie subalgebra of $\hdg$, whose complexification
contains  $\f_H^{0}$. This contradicts the minimality property of $\hdg$ and
therefore proves the Lemma.
\end{proof}

\begin{rem}
\label{tracecenter} It follows from Remarks \ref{function} that
$$\Tr_{E_{\C}}(\f_H^{0})=\{\frac{d}{2}-n_{\sigma}\}_{\sigma \in\Sigma_E}$$
lies in $E_{-}\otimes_{\Q}\C$. Applying Lemma \ref{traceminimal}, we conclude
that $\Tr_{E}(\hdg)\subset E_{-}.$
\end{rem}

\begin{thm}
\label{normal}
 Suppose that $E$ is a CM field that is normal over $\Q$ and fix
a field embedding $E \hookrightarrow \bar{\Q}\subset\C$. Let $W$ be the
$\Q[\Gal(E/\Q)]$-submodule of $X_E$ generated by the function
$h(\sigma):=\{\frac{d}{2}-n_{\sigma}\}_{\sigma \in\Sigma_E}$. Then
$$\dim_{\Q}(\Tr_E(\cc^0))=\dim_{\Q}(\Tr_{E}(\hdg))=\dim_{\Q}(W),$$
and therefore, $\dim_{\Q}(\cc^0) \geq \dim_{\Q}(W)$.
If $W=X_E$ then \[\Tr_E(\cc^0)=\Tr_{E}(\hdg)= E_{-}.\]
\end{thm}

\begin{proof} Let $\q\subset E_{-}$ be the smallest $\Q$-vector subspace such
that $\q_{\C}$ contains the function $h$. By Lemma \ref{CMdim},
$$\dim_{\Q}(\q)=\dim_{\Q}(W).$$
The minimality properties of $\hdg$ imply that
$$\q=\Tr_{E}(\hdg)=\Tr_E(\cc^0).$$ This implies
that
$$\dim_{\Q}(\cc^0)=\dim_{\Q}(\Tr_{E}(\hdg))=\dim_{\Q}(\q)=\dim_{\Q}(W).$$

\end{proof}

\begin{thm}
\label{simplefactor} Suppose that $E$ is a CM field that is normal over $\Q$.

\begin{itemize}
\item[(i)] If the center $\CC_Z$ of $\End^0(Z)$ is a field then $Z$ contains
a simple abelian subvariety of dimension $\ge \dim_{\Q}\Tr_E(\hdg_Z)$.

\item[(ii)] If $\Tr_{E}(\hdg)= E_{-}$ then $Z$ contains a simple abelian
subvariety of dimension $\ge \max(E)$.
\end{itemize}
\end{thm}

\begin{proof}

We may assume that $\Tr_{E}(\hdg)\ne \{0\}$.

(i) Suppose that the center $\CC_Z$ of $\End^0(Z)$ is a
field. Since the center of $\End^0(Z)$ is a field, there exists a simple
complex abelian (sub)variety $T$ (of $Z$) such that $Z$ is isogenous to a
self-product $T^m$ of $T$; in particular, $\End^0(Z)$ is the matrix algebra of
size $m$ over $\End^0(T)$, which implies that $\CC_Z =\CC_T$.

If $\c^{0}$ is the center of $\hdg_Z$ then $\c^{0}\subset \CC_Z=\CC_T$ and
$\dim_{\Q}(\c^{0}) \ge \dim_{\Q}\Tr_E(\hdg_Z)$. If $\CC_Z=\CC_T$ is totally real
then the center $\c^0$ of $\hdg_Z$ is zero, which is not the case. So, $\CC_Z$
is a CM-field and
$$\dim_{\Q}(\c^{0}) \le \frac{1}{2}[\CC_Z:\Q]= \frac{1}{2}[\CC_T:\Q].$$
This implies that
$$\frac{1}{2}[\CC_T:\Q] \ge \dim_{\Q}(\c^{0}) \ge \dim_{\Q}\Tr_E(\hdg_Z).$$
Since $[\CC_T:\Q]\le 2\dim(T)$, we conclude that $\dim(T) \ge
\dim_{\Q}\Tr_E(\hdg_Z)$. This proves (i).

\smallskip

(ii) Let us assume that $\Tr_E(\hdg)=E_{-}$. Clearly, the center $\c^0$ of $\hdg_Z$ satisfies
 $$\dim_{\Q}(\c^0) \ge \dim_{\Q}(E_{-})=\frac{1}{2}[E:\Q].$$
The Poincar\'e reducibility theorem implies that there exist abelian
subvarieties $Z_1, \dots , Z_r$ of $Z$ such that the natural morphism
$$\pi:\prod Z_j \to Z, \ \{z_j\}\mapsto \sum_j z_j $$
is an isogeny, $\Hom(Z_i,Z_j)=\{0\}$ for all $i \ne j$ and each $\End^0(Z_j)$
is a simple (but not necessarily central) $\Q$-algebra and its center is a
field.. In particular,
$$\End^0(Z)=\oplus_j \End^0(Z_j).$$

The morphism $\pi$ implies the $\End^0(Z)$-equivariant isomorphism of rational
$\Q$-structures
$$\H^1(Z,\Q)=\oplus_j \H^1(Z_j,\Q).$$
In particular, this splitting is $E$-invariant and we have the embeddings $E
\hookrightarrow \End(Z_j)$, whose ``direct sum" is the $E \hookrightarrow
\End^0(Z)$.  We have
$$d(Z,E)=\sum_j d(Z_j,E); n_{\sigma}(Z,E)=\sum_j n_{\sigma}(Z_j,E) \ \forall
\sigma\in \Sigma_E.$$ It follows that the function
$$h_{Z,E}: \Sigma_E\to \Q, \ \sigma \mapsto \frac{1}{2}d(Z,E)-
n_{\sigma}(Z,E)$$ coincides with the sum $\sum_j h_{Z_j,E}$ where
$$h_{Z_j,E}: \Sigma_E\to \Q, \ \sigma \mapsto \frac{1}{2}d(Z_j,E)-
n_{\sigma}(Z_j,E)$$ is the corresponding function attached to $Z_j$. Clearly,
$h$ and all $h_j$ belong to $X_E$.

Let $W_j$ be the $\Q[\Gal(E/\Q)]$-submodule of $X_E$ generated by $h_j$. Since
$h=\sum_j h_j$, the $\Q[\Gal(E/\Q)]$-submodule $\sum_j W_j$ contains $h$.
Let $W$ be the $\Q[\Gal(E/\Q)]$-submodule of $X_E$ generated by $h$. By Theorem
\ref{normal}, $\dim_{\Q}(W)=\dim_{\Q}(E_{-})$. Since
$\dim_{\Q}(E_{-})=\dim_{\Q}(X_E)$, we conclude that
$\dim_{\Q}(W)=\dim_{\Q}(X_E)$  and therefore $X_E=W$. In other words,
$X_E$ coincides with its $\Q[\Gal(E/\Q)]$-submodule
generated by $h$. It follows that $X_E=\sum_j W_j$. This implies that if
$W^{\prime}$ is a simple $\Q[\Gal(E/\Q)]$-submodule of $X_E$ then it is
isomorphic to a certain $\Q[\Gal(E/\Q)]$-submodule of $W_j$ for some $j$; in
particular, $1 \le \dim_{\Q}(W^{\prime}) \le \dim_{\Q}(W_j)$. On the other
hand, by Theorem \ref{normal}, $\dim_{\Q}(W_j)= \dim_{\Q}\Tr_E(\hdg_{Z_j})$. By
the already proven case (i), $Z_j$ contains a simple abelian subvariety $T_j$
with $\dim(T_j) \ge \dim_{\Q}\Tr_E(\hdg_{Z_j})$. It follows that $$ \dim(T_j)
\ge\dim_{\Q}\Tr_E(\hdg_{Z_j})\ge \dim_{\Q}(W^{\prime}).$$ Clearly, $T_j$ is an
abelian subvariety of $Z$, since $Z_j$ is an abelian subvariety of $Z$. Now, if
we choose $W^{\prime}$ with $\dim_{\Q}(W^{\prime})=\max(E)$ then we get
$\dim(T_j) \ge \max(E)$.
\end{proof}


\begin{sect}
\label{product} Let $t$ be a positive integer and suppose that for each
positive $j\le t$ we are given the following data.

\begin{itemize}
\item A number field $E_i$ that is normal over $\Q$; we fix an embedding
$E_j\hookrightarrow\C$ and consider $E_j$ as the subfield of $\C$.

\item A complex abelian variety $Z_i$ of positive dimension.

\item An embedding $E_j \hookrightarrow \End^0(Z_j)$ that sends $1$ to the
identity automorphism of $Z_j$.
\end{itemize}

Let us consider the corresponding numbers
$$d_j:=d(Z_j,E_j)=\frac{2\dim(Z_j)}{[E_j:\Q]}$$
and functions
$$\Sigma_{E_j} \to \Z_{+}, \ \sigma \mapsto n^{\{j\}}_{\sigma}:=n_{\sigma}(Z_j,
E_j)=\dim_{\C}\Omega^1(Z_j)_{\sigma}.$$

Let us consider the products ${Z}=\prod_{j=1}^t Z_j$ and $\E =\prod_{j=1}^t
E_j=\oplus_{j=1}^t E_j$. Clearly, ${Z}$ is a complex abelian variety and $\E$
is is a finite-dimensional semisimple commutative $\Q$-algebra that admits a
natural embedding
$$\E \hookrightarrow \End^0({Z})$$
that sends $1\in \E$ to $1_{\mathcal{Z}}$. The natural (K\"unneth) isomorphism
$$\H_1({Z},\Q)=\oplus_{j=1}^t \H_1(Z_j,\Q)$$
is an isomorphism of rational Hodge structures; in particular, each
$\H_1(Z_j,\Q)$ is a $\MT_Z$-invariant $\Q$-vector space of $\H_1({Z},\Q)$.
Clearly, $\H_1({Z},\Q)$ carries the natural structure of $\E$-module and
$$\E\subset \End^0(Z)\subset \End_{\Q}(\H_1({Z},\Q)).$$
It is also clear that
$$\End_{\E}(\H_1({Z},\Q))=\oplus_{j=1}^t \End_{E_j}(\H_1({Z_j},\Q))$$ and
$$\hdg_Z \subset \oplus_{j=1}^t \End_{E_j}(\H_1({Z_j},\Q)).$$
In particular, the elements of $\End_{\E}(\H_1({Z},\Q))$ are the $t$-tuples
$\{u_j\}_{j=1}^t$ with $u_j \in \End_{E_j}(\H_1({Z_j},\Q))$. Clearly, the map
$$\{u_j\}_{j=1}^t \mapsto u_j \mapsto \Tr_{E_j}(u_j) \in E_j$$
is the homomorphism of $\Q$-Lie algebras
$$\End_{\E}(\H_1({Z},\Q)) \to E_j,$$
which we continue to denote by $\Tr_{E_j}$. Since $E_j$ is the commutative Lie
algebra, $\Tr_{E_j}$ kills the semisimple part of $\hdg_Z$. On the other hand,
the restriction of $\Tr_{E_j}$ to $\E$ coincides with the composition of the
projection map
$$\E=\oplus_{j=1}^t E_j\to E_j$$ and multiplication by $d_j$.
 Let $d$ be the least common multiple of all $d_j$'s. Then the $\Q$-linear map
 $$\Tr_{\E}:\oplus_{j=1}^t \End_{E_j}(\H_1({Z_j},\Q)) \to \oplus_{j=1}^t E_j, \
\{u_j\}_{j=1}^t \mapsto \left\{\frac{d}{d_j}\Tr_{E_j}(u_j))\right\}_{j=1}^t$$
kills the semisimple part of $\hdg_Z$ and acts on $\E$ as multiplication by $d$.
It follows that
 $$\Tr_{\E}(\hdg_Z)=\Tr_{\E}(\cc^{0}_Z)\subset \E;$$
 in addition, if $\CC_Z \subset \E$ then $\cc^{0}_Z\subset \E$ and
$\Tr_{\E}(\cc^{0}_Z)=\cc^{0}_Z$, which implies that
$$\Tr_{\E}(\hdg_Z)=\cc^{0}_Z\subset \E.$$
As above, $\Tr_{\E}(\hdg_Z)$ coincides with the smallest $\Q$-vector subspace
$\q\subset \E$ such that $\q_{\C}$ contains $\Tr_{\E}(\f_{H,Z}^0)$  where
$$\f_{H,Z}^0=\f_{H,Z}+\frac{1}{2}\I_Z \in \End_{\Q}(\H_1(Z,\Q))\otimes_{\Q} \C.$$
On the other hand, one may easily check that $\Tr_{\E}(\f_{H,Z}^0)$ corresponds
(via $\kappa_{\E}$) to the function $h:\Sigma_E \to \Q\subset \C$ that
coincides with $$h_j:\Sigma_{E_j}\to \Q, \ \sigma\mapsto
\frac{d}{2}-\frac{d}{d_j}n_{\sigma}(Z_j)$$ on $\Sigma_{E_j}$.
\end{sect}

\begin{thm}
\label{TRACECENTER} We keep the notation and assumptions of the previous
subsection. Suppose that $E_t$ contains all $E_j$'s and for all $j$ and  each
$\sigma_j \in \Sigma_j$ we have $h_j(\sigma_j)=\sum_{\sigma} h_t(\sigma)$ where
the sum is taken across all $\sigma E_t\hookrightarrow \C$, whose restriction
to $E_j$ coincides with $\sigma_j$. Then
$$\Tr_{\E}(\hdg_Z)=\{ (e_j)_{j=1}^t \in \E\mid e_t\in \Tr_{E_t}(\hdg_{Z_t}),\  e_j=\Tr_{E_t/E_j}(e_t) \ \forall
j\}.$$ In particular, $\dim_{\Q}(\q)=\dim_{\Q}(\q_t)$.
\end{thm}

\begin{proof}
One has only to notice that $\Tr_{\E}= (d/d_t)\Tr_{E_t}$ on $\hdg_{Z_t}\subset
\End_{E_t}(\H_1(Z_t,\Q))$ and apply Theorem \ref{TRACEHG}.

\end{proof}

\section{Proof of Main Results}
\label{sectmain}
We keep all notation and assumptions of Section \ref{MT}.

\begin{sect}
Suppose that  $n\ge 2$ is an integer, $p$ is a prime that does not divide $n$.
Let $r$ be a positive integer and $q=p^r$. Suppose that $E=\Q(\zeta_q)$ and
 $$d(Z,E)=n-1.$$
 It is well known that $\Gal(E/\Q)=(\Z/q\Z)^{*}$ where $a+q\Z \in (\Z/q\Z)^{*}$ corresponds to
 the field
 automorphism
 $$s_a: \Q(\zeta_q) \to \Q(\zeta_q), \ \zeta_q \to \zeta_q^a.$$
 It is also well-known that the complex conjugation $c_0$ coincides with
 $s_{-1}$.

 Clearly,  $\Sigma_E$ coincides with the
 set of
 embeddings
 $$\sigma_a: \Q(\zeta_q) \to \Q(\zeta_q)\subset \C$$
 that send $\zeta_q$ to $\zeta_q^{-a}$ with $a+q\Z \in (\Z/q\Z)^{*}$.
 It is also clear that
 $$\bar{\sigma}_a=\iota\sigma_a=\sigma_{-a}=\sigma_{q-a}$$ and
 $$s_b(\sigma_a)=\sigma_{ab} \ \forall \ a,b.$$

\end{sect}

 \begin{thm}
 \label{inttrace}
Suppose that $n_{\sigma_a}=[na/q]$ for all $a$ with $1 \le a<q, \ (a,p)=1$.
Then $\Tr_{E}(\cc^0)=\Tr_{E}(\hdg)= E_{-}$.
 \end{thm}

\begin{proof}

The following statement will be proven in Section \ref{arithm} (See Theorem
\ref{thm:gen1}). \\
{\itshape Let us consider the $\Q$-vector space   of functions
 \[V_{\Q}:=\{g:\left(\Z/q\Z\right)^{\times}\ra \Q \mid g(q-a)=-g(a),
 \forall \ a+q\Z\}\] provided with the natural structure of a
 $(\Z/q\Z)^{*}$-module. Let $n$ be a positive integer that is not
 divisible by $p$.  Let us consider the function $h$ on $(\Z/q\Z)^{*}$
 defined by $h(a+q\Z)=\frac{n-1}{2}-[\frac{na}{q}]$ for $1\leq a \leq
 q-1$ and $p\nmid a$. Then $h \in V_{\Q}$ and the $(\Z/q\Z)^{*}$-submodule generated
 by $h$ coincides with $V_{\Q}$.}

Clearly, $V_{\Q}=X_E$. Now Theorem \ref{inttrace} becomes an immediate corollary of
Theorem \ref{normal} combined with the above result.
\end{proof}

Now we are ready to prove our main theorems listed in Section \ref{intro}.

 {\sl Proof of Theorem} \ref{main0}. By \cite[p. 355 and Remark 4.13 on p.
356]{ZarhinM}, \cite[Remark 5.14 on p. 383]{ZarhinPisa}  there exists an
embedding $\Q(\zeta_q)\hookrightarrow \End^0(J^{(f,q)})$ such that
$d(J^{(f,q)},\Q(\zeta_q))=n-1$ and $n_{\sigma_a}=[na/q]$ for all $a$ with
$1\leq a \leq q-1$. Now the result follows from  Theorem \ref{inttrace}.

{\sl Proof of Theorem} \ref{mainD}. If $p$ is odd then $\Gal(\Q(\zeta_q)/Q)$ is
a cyclic group of order $(p-1)p^{r-1}$.  If $p=2$ and $q\ge 4$ then
 $\Gal(\Q(\zeta_q)/Q)=\langle c\rangle \times H$ where $c$ is the complex conjugation and $H$
 is a cyclic group of order $2^{r-1}$. Now the result follows from Theorem
 \ref{main0} combined with Theorem \ref{simplefactor} and Examples \ref{maxcycl}.

{\sl Proof of Theorem} \ref{main}. We know that
$\End^0(J^{(f,q)})=\Q(\zeta_q)=E$. By the last assertion of Subsect.
\ref{hodge},  $\cc^0 \subset E_{-} \subset E$ and therefore
$\Tr_E(\cc^0)=\cc^0$. By Theorem \ref{main0}, $\Tr_E(\cc^0)=E_{-}$. This
implies that
$$\cc^0=\Tr_E(\cc^0)=E_{-}=\Q(\zeta_q)_{-}.$$
Since $E_{-}$ coincides with the $\Q$-Lie algebra of {\sl connected}  $\TT_{\Q(\zeta_q)}$, we conclude that the center $\ZZ$ of $\Hdg(J^{(f,q)})$ contains $T_E$.
But we proved in Subsect. \ref{HdgCenter} that $\ZZ\subset \TT_E$. It follows that $\ZZ=\TT_E=\TT_{\Q(\zeta_q)}=\U_q$.

{\sl Proof of Theorem} \ref{mainfull}. We know that
$d_j=d(J^{(f,p^j)},\Q(\zeta_{p^j}))=n-1$ for all $j\leq r$. The least
common multiple $d$ of all $d_j$ is also $n-1$. It follows that the
function $h_j: \Sigma_{E_j}\to \Q$ is defined by
\[ h_j(a+p^j\Z) =
\frac{n-1}{2}-\left[\frac{na}{p^j}\right], \quad \forall 1\leq a\leq p^j-1,
p\nmid a. \] The following relations between the functions $h_j$ will be proved
in Section \ref{lastsection} (See Corollary \ref{dependR}).

{\itshape Let us identify (in the
usual way) $G=(\Z/p^r\Z)^{*}$ with the Galois group $\Gal(\Q(\zeta_{p^r}/\Q)$
and let us consider its subgroup
$$G_j=\Gal(\Q(\zeta_{p^r})/\Q(\zeta_{p^j}))\subset \Gal(\Q(\zeta_{p^r}/\Q)=G.$$
Then for each $a \in (\Z/p^r\Z)^{*}$,
$$h_j(a \bmod p^j)=\sum_{b\in G_j}h_r(ab).$$ }
 The Theorem now follows from Theorem \ref{TRACECENTER}.



\section{Fourier coefficients}
\label{arithm}

\begin{sect}

  Throughout this section, $p$ is a prime, $q=p^r$ is a power of $p$,
  and $n$ is a positive integer that is {\sl not} divisible by $p$.

  As usual,
\[\ST^1:=\{z\in\C\mid z\bar{z}=1\} \subset \C^{*}\subset \C.\]
Given a finite  group $G$, its group of characters $\widehat{G}$ is the group $\Hom(G,\ST^1)$. (If $G$ is commutative then
$\widehat{G}$ is called the dual of $G$.) Let $\K$ be a field that is either
$\Q$ or $\C$.  Recall that the regular representation $R_{\K}$ of $G$ over $\K$
is the space of $\K$-valued function on $G$, where an element $a\in G$ acts on
a function $f$ by $(af)(b)=f(ba), \forall b \in G$. Clearly
$R_{\C}=R_{\Q}\otimes_{\Q}\C$.

Suppose $G=\gp$. We write $V_{\K}$ for the subrepresentation of $R_{\K}$
consisting all ``odd'' functions on $\gp$. Namely,
\[V_{\K}:=\{f:\left(\Z/q\Z\right)^{\times}\ra \K \mid f(q-a)=-f(a),
\forall \ a+q\Z\}.\] By definition, $V_{\K}=\{0\}$ if $q=p=2$, and
$\dim_{\K}V_{\K}=\varphi(q)/2$ otherwise.

Given a real number $x$, we write $[x]$ for the largest integer less
  or equal to $x$. Now consider the function $h_r$ defined by
$h_r(a)=\frac{n-1}{2}-[\frac{na}{q}]$, where $1\leq a \leq q-1$ and $ p\nmid
a$. Since $p\nmid n$, we have
\[ \left[\frac{na}{q}\right]+\left[\frac{n(q-a)}{q}\right]= n-1 \qquad
\forall\: 0<a<q \text{ with }p\nmid a.\] Hence $h_r\in V_{\Q}\subset
V_{\C}$.  We also note that $h_r=0$ if and only if either $q=2$ or $n=1$.

\end{sect}

The following assertion was used in the proof of Theorem \ref{inttrace}.

\begin{thm}\label{thm:gen1}
  Let $p$ be a prime, $q=p^r$, $n\geq 2$ and $p\nmid n$. Let $\K$ be
  either $\Q$ or $\C$. The function $h_r$ generates the $\K[\gp]$-module
  $V_{\K}$.
\end{thm}

\begin{proof}

  If $q=2$, the vector space $V_{\K}=\{0\}$; if $q=4$, then $\dim_{\K}
  V_{\K}=\phi(4)/2=1$ and $h \neq 0$, hence the theorem is trivial in
  these cases.  Thus we further assume that either $p$ is odd or $p=2$
  and $q=2^r\geq 8$.

Let $W_{\K}$ be the submodule of $V_{\K}$ generated by $h_r$. Clearly
$W_{\C}=W_{\Q}\otimes \C$. Then $h_r$ generates $\Q[\gp]$-module
$V_{\Q}$ if and only if the same holds true if we replace $\Q$ with
$\C$. From now on we work exclusively over the
  field of complex numbers.

  Let $G$ be a finite commutative group. The regular representation
  $R_{\C}$ of $G$ decomposes into a direct sum of 1-dimensional
  irreducible subrepresentations generated by the characters (see
  \cite[Corollary 2.18]{Rep}):
  \[ R_{\C}=\oplus_{\chi\in \widehat{G}} \C\cdot \chi.\] Recall that
  $V_{\C}$ is a subrepresentation of the regular representation for
  $G=\gp$. Hence \[V_{\C}=\oplus \C\cdot \chi,\] where we sum over all
  characters $\chi \in \widehat{G}\cap V_{\C}$.

  A character $\chi$ lies in $V_{\C}$ if and only if
  $\chi(-1)=-1$. These characters are mutually orthogonal under the
  inner product
\[\langle g_1, g_2\rangle:=\frac{1}{\varphi(q)}\sum_{\substack{1\leq a\leq
    q-1 \\ (a,p)=1}}g_1(a)\overline{g_2(a)}, \qquad \forall g_1,
g_2\in V_{\C}. \] Clearly, $\langle\chi, \chi \rangle=1$. It follows that the
set
\[B:=\widehat{G}\cap V_{\C} =\{\; \chi \in \widehat{G} \mid
\chi(-1)=-1\;\}\] forms an orthonormal basis of $V_{\C}$. In particular, there
are exactly $\varphi(q)/2$ characters that are in $B$ (which could also be seen
from the fact that $\sum_{\chi\in \widehat{G}} \chi(-1)=0$ ). We label the
characters in $B$ as $\chi_j$ for $1 \leq j \leq \varphi(q)/2$.

Every function $f\in V_{\C}$ may be uniquely written as a linear combination of
characters $\sum c_{j}\chi_j$, where
\[c_{j}=\langle f,
\chi_j \rangle =\frac{1}{\varphi(q)}\sum_{\substack{1\leq a\leq q-1 \\
    (a,p)=1}}f(a)\overline{\chi_j(a)}.\]
If $c_i=0$ for some $i$, then the $\gp$-submodule generated by $f$ is contained
in the proper submodule $\oplus_{j\neq i}\C\cdot \chi_j$. Thus if $f$ generates
the $\gp$-module $V_{\C}$, it is necessary that $c_j\neq 0$ for all $j$. We show that this is also a sufficient condition. It
suffices to show that each $\chi_i$ lies in the $\gp$-submodule generated by
$f$. For each $j\neq i$, we choose an element $a_{ij}\in \gp$ such that
$\chi_i(a_{ij}) \neq \chi_j(a_{ij})$. Let $T_i$ be the element in the group
$\C$-algebra $\C[\gp]$ defined by
\[T_i=\frac{1}{c_i}\prod_{\substack{j=1\\j\neq i}}^{\varphi(q)/2} \frac{a_{ij}-\chi_j(a_{ij})}{\chi_i(a_{ij})-\chi_j(a_{ij})}.\]
Clearly $T_i\chi_j=0$ for all $j\neq i$, and $T_i\chi_i=\chi_i/c_i$. Then $T_i
f =\chi_i$. Since $i$ is arbitrary, we conclude that all $\chi_i\in B$ lie in
the $\gp$-submodule generated by the function $f$. Therefore, the theorem
follows from the following lemma.
\end{proof}

\begin{lem}
\label{lastlemma}
  Let $p$ be a prime, $q=p^r$, $n\geq 2$ and $p\nmid n$.  Let $\chi:
  \gp\ra \ST^1\subset \C$ be a character of $\gp$ such that
  $\chi(-1)=-1$. Then the sum $\sum h_r(a)\overline{\chi(a)} \neq 0$,
  where we sum over all integers $a$ such that $1\leq a \leq q-1$ and $p\nmid
  a$.
\end{lem}

We will prove Lemma \ref{lastlemma} in Section \ref{lastsection}.

\section{Explicit formulas}
\label{lastsection}

\begin{proof}[Proof of Lemma \ref{lastlemma}]
  Using Fourier expansion, one sees that for $x\not\in \Z$,
  \[ x-[x]-\frac{1}{2}=\sum_{\substack{m\in \Z\\ m \neq 0}}
  -\frac{e(mx)}{2\pi i m}\; ,\]
where $e(x):=e^{2\pi ix}$. If we let $s(h,\chi)$ to be the sum $\sum
h(a)\overline{\chi(a)}$, and regard $\chi$ as Dirichlet character
(i.e., we put $\chi(a)=0$ if $p\mid a$)
, we then get
\[
\begin{split}
  s(h_r,\chi) &=\sum_{a=0}^{q-1}
  \left(\frac{n-1}{2}-\left[\frac{na}{q}\right]\right)\overline{\chi(a)} \\
   &=\sum_{a=0}^{q-1}\left(\frac{n-1}{2}-\frac{na}{q}+\frac{1}{2}-\sum_{\substack{m\in
         \Z\\ m \neq 0}}\frac{e(mna/q)}{2\pi i
       m}\right)\overline {\chi(a)}\\
   &=\frac{n}{2}\sum_{a=0}^{q-1}\overline {\chi(a)}
   -\frac{n}{q}\sum_{a=0}^{q-1} a\;\overline
   {\chi(a)}-\sum_{\substack{m\in \Z \\m \neq 0}}\frac{1}{2\pi i m } \sum_{a=0}^{q-1}e(mna/q)\overline{\chi(a)}\\
   &=-\frac{n}{q}\sum_{a=0}^{q-1} a\;\overline
   {\chi(a)}-\sum_{\substack{m \in \Z\\ m \neq 0}}\frac{1}{2\pi i m } \sum_{a=0}^{q-1}e(mna/q)\overline{\chi(a)}
\end{split}
\]
where we used the fact that $\sum_{a=0}^{q-1}\overline{\chi(a)}=0$. Since $n$
and $q$ are coprime, $na\bmod q$ runs through the list of all residue classes
modulo $q$ when $a$ does so. Then
\begin{equation}  \label{eq:head}
\begin{split}
s(h_r,\chi)&=-\frac{n}{q}\sum_{a=0}^{q-1} a\;\overline
   {\chi(a)}-\sum_{\substack{m \in \Z\\ m \neq 0}}\frac{\chi(n)}{2\pi
   i m } \sum_{a=0}^{q-1}e(mna/q)\overline{\chi(na)}\\
&=-\frac{n}{q}\sum_{a=0}^{q-1} a\;\overline
   {\chi(a)}-\sum_{\substack{m \in \Z\\ m \neq 0}}\frac{\chi(n)}{2\pi i m } \sum_{a=0}^{q-1}e(ma/q)\overline{\chi(a)}
\end{split}
\end{equation}
Let $\tau_q(\chi)$ denote the Gauss sum $\sum e(a/q)\chi(a)$, also set
$S_q(\chi)=\sum a\,\chi(a)$ and $c_{\chi}(m)=\sum e(ma/q)\chi(a)$,   where all
sums are taken from $a=0$ to $q-1$. We then have

\begin{equation}
  \label{eq:neck}
  s(h_r,\chi)=-\frac{n}{q}S_q(\bar{\chi})-\sum_{\substack{m\in \Z \\ m
      \neq 0 }}\frac{\chi(n)c_{\bar{\chi}}(m)}{2\pi
  i m}
\end{equation}
Clearly, (\ref{eq:neck}) works for all $n$ such that $p\nmid n$. In particular,
if we set $n=1$, then $h_r=0$, hence $s(h_r,\chi)=0$. That is
\[-\frac{1}{q}S_q(\bar{\chi})-\sum_{\substack{m\in \Z \\ m
      \neq 0 }}\frac{c_{\bar{\chi}}(m)}{2\pi
  i m}=0
 \]
Combining with (\ref{eq:neck}) we get
\begin{equation}
  \label{eq:shoulder}
s(h_r,\chi)=-\frac{n}{q}S_q(\bar{\chi})+\frac{\chi(n)}{q}S_q(\bar{\chi})=\frac{1}{q}(\chi(n)-n)S_q(\bar{\chi})
\end{equation}
When $n\geq 2$,  $\chi(n)-n\neq 0$ since $\abs{\chi(n)}=1$. Thus the lemma
follows if we show that $S_q(\bar{\chi})\neq 0$ for any character $\chi$ with
$\chi(-1)=-1$. We prove this by cases.

\bigskip

\noindent Case 1. Assume that $\chi$ is a primitive Dirichlet character modulo
$q=p^r$.  By \cite[Theorem 9.7]{Vau},
\[c_{\bar{\chi}}(m)=\chi(m)\tau_q(\bar{\chi}).\]
It follows that
\begin{equation}
  \label{eq:main}
  \begin{split}
S_q(\bar{\chi})&=-q\sum_{\substack{m \in \Z \\
    m \neq 0}}\frac{c_{\bar{\chi}}(m)}{2\pi i m }
  =-\frac{q\tau_q(\bar{\chi})}{2\pi
    i } \sum_{\substack{ m \in \Z\\ m \neq 0}} \frac{ \chi(m)}{m} \\
  &=-\frac{q\tau_q(\bar{\chi})}{\pi
    i } \sum_{m=1}^{\infty} \frac{ \chi(m)}{m}
  =-\frac{q\tau_q(\bar{\chi})}{\pi i }L(1,\chi)
  \end{split}
\end{equation}
where we used the fact that $\chi(-1)=-1$.

It is well known \cite[Theorem 9.7]{Vau} that the absolute value of the Gauss
sum $\abs{\tau_q(\chi)}=\sqrt{q}$ for all primitive characters $\chi$. In
particular $\tau_q(\bar{\chi}) \neq 0$. We get
\begin{equation}
  \label{eq:l1}
    L(1,\chi)=-\frac{\pi i S_q(\bar{\chi})}{q \tau_q(\bar{\chi})}=\frac{\pi
    i }{q^2}\,\tau_q(\chi)S_q(\bar{\chi})
\end{equation}
since
$\tau_q(\bar{\chi})=\chi(-1)\overline{\tau_q(\chi)}=-\overline{\tau_q(\chi)}$.

It is a classical result that $L(1,\chi)\neq 0$ for all nontrivial Dirichlet
characters modulo $q$ (see \cite[ Theorem 2, Chapter 16]{Ire}). Hence
$S_q(\bar{\chi})\neq 0$.

The above closed form (\ref{eq:l1}) of $L(1,\chi)$ for $\chi(-1)=-1$ is
actually also classical. See \cite[ Theorem 9.9]{Vau}.

\bigskip

\noindent Case 2. Assume that $\chi$ is induced by a primitive character $\xs$
modulo $d$ where $d=p^{r_d}$ with $0<r_d<r$. Since both $d$ and $q$ are powers
of $p$, $\chi(a)=\xs(a)$ for all $0\leq a\leq q-1$.  If we write $a=xd+y$ with
$0\leq x\leq q/d-1$ and $0\leq y \leq d-1$, we then have
\begin{equation}
  \label{eq:sx}
\begin{split}
S_q(\bar{\chi})=&\sum_{a=0}^{q-1}
a\,\overline{\chi(a)}=\sum_{x=0}^{q/d-1}\sum_{y=0}^{d-1}
(xd+y)\,\bxs(xd+y)\\
&=\sum_{x=0}^{q/d-1}\sum_{y=0}^{d-1}xd\,\bxs(y)
+\sum_{x=0}^{q/d-1}\sum_{y=0}^{d-1}
y\,\bxs(y)\\
&=\frac{q}{d}\sum_{y=0}^{d-1}y\bxs(y)=\frac{q}{d}\,S_d(\bxs)
\end{split}
\end{equation}
By Case 1, $S_d(\bxs)\neq 0$. Thus $S_q(\bar{\chi})\neq 0$.
\end{proof}

\begin{cor}
\label{rdFourier} Let $p$ be a prime, $q$ power of $p$, $n$ an integer coprime
to $p$. Suppose that $\chi$ is a Dirichlet character mod $q$ that is induced by
a character $\xs$ mod $d$ for some $d\mid q$. Then
$\sum_{a=0}^{q-1}\left[na/q\right]\bar{\chi}(a)=\sum_{b=0}^{d-1}\left[nb/d\right]\bxs(b).$
In particular, if  $c_{\chi}^{(r)}$ is the coefficient of $h_r$ with respect to
$\chi$ and $c_{\xs}^{(d)}$ is the coefficient of $h_d$ with respect to $\xs$
then
$$\varphi(q)c_{\chi}^{r}=\varphi(d)c_{\xs}^{(d)}.$$
\end{cor}

\begin{proof}
  First assume that $\xs$ is primitive mod $d$. By previous
  calculations, if we substitute (\ref{eq:sx}) into
  (\ref{eq:shoulder}), then
\[s(h_r,\chi)=\frac{1}{q}(\chi(n)-n)\frac{q}{d}\,S_d(\bxs)=\frac{1}{d}(\chi(n)-n)S_d(\bxs).\]
Applying (\ref{eq:shoulder}) again,
\[\frac{1}{d}(\chi(n)-n)S_d(\bxs)= \sum_{b=0 }^{d-1}\left(\frac{n-1}{2}-\left[\frac{nb}{d}\right]\right)\bxs(b)\]
It follows that
\[\sum_{a=0}^{q-1}\left(\frac{n-1}{2}-\left[\frac{na}{q}\right]\right)\bar{\chi}(a)=s(h_r,\chi)=\sum_{b=0}^{d-1}\left(\frac{n-1}{2}-\left[\frac{nb}{d}\right]\right)\bxs(b).\]
Equivalently,
\[\sum_{a=0}^{q-1}\left[\frac{na}{q}\right]\bar{\chi}(a)=\sum_{b=0}^{d-1}\left[\frac{nb}{d}\right]\bxs(b).\]
If $\xs$ is not primitive, then we reduce further
  to obtain a primitive
  character $\xs_0$ mod $d'$  for some $d'\mid d$ such that $\xs_0$
  induces both $\xs$ and $\chi$. Both sides of the
  desired equality now equal to $\sum_{c=0}^{d'-1}\left[nc/d'\right]\bxs_0(c)$.
\end{proof}

The following assertion was used in the proof of Theorem \ref{mainfull}.
\begin{cor}
\label{dependR} Let $j \le r$ be a positive integer. Let us identify (in the
usual way) $G=(\Z/p^r\Z)^{*}$ with the Galois group $\Gal(\Q(\zeta_{p^r}/\Q)$
and let us consider its subgroup
$$G_j=\Gal(\Q(\zeta_{p^r})/\Q(\zeta_{p^j}))\subset \Gal(\Q(\zeta_{p^r}/\Q)=G.$$
Then for each $a \in (\Z/p^r\Z)^{*}$,
$$h_j(a \bmod p^j)=\sum_{b\in G_j}h_r(ab).$$
\end{cor}

\begin{proof}
If $\chi$ is a character that does not kill $G_j$ (i.e., is not induced from
$(\Z/p^j\Z)^{*}$ then $\sum_{b\in G_j}\chi(b)=0$. If $\chi$ is a character that
 kills $G_j$ (i.e., is induced from $(\Z/p^j\Z)^{*}$) then $\sum_{b\in
G_j}\chi(b)=\#(G_j)=\varphi(p^r)/\varphi(p^j)$. Now the result follows from the
last assertion of Corollary \ref{rdFourier}.

\begin{exs}
If $q=p\equiv 3 \pmod 4$, one sees that the Legendre symbol $\chi(a)=(a/p)$
satisfies $(-1/p)=-1$. It is also known (\cite[Theorem 1, page 75]{Ire},
\cite[formula (19), page 51]{dav})  that
\begin{gather}
  \label{eq:g1}
  \tau_p(\bar{\chi})=\tau_p(\chi)=\sqrt{-p}\:,\\
  \label{eq:g2}
  S(\bar{\chi})=S(\chi)=-p\, \h_p\:,
\end{gather}
where $\h_p$ denotes the class number of the imaginary quadratic field
$\Q(\sqrt{-p})$. Combining (\ref{eq:l1}) and (\ref{eq:g2}), we get
\begin{equation}
    \label{eq:g3}
  L(1,\chi)=\pi \h_p/\sqrt{p}\:.
\end{equation}
Alternatively, one could also deduce (\ref{eq:g3}) by combining the formulas
\cite[(17), page 50]{dav}, \cite[(3), page 45]{dav}, and the formula at the
bottom of page 52 of \cite{dav}.

Therefore, in the case $q=p\equiv 3 \pmod 4$, one has a closed formula
\begin{equation}
  \label{eq:p3}
 \sum_{1 \leq a \leq p-1}\:
  h_1(a)\left(\frac{a}{p}\right)=
  \left(n-\left(\frac{n}{p}\right)\right)\h_p \,.
\end{equation}

\end{exs}

\end{proof}

\section{Semilinear algebra}
\label{semilinear}

Let $Q$ be a field of characteristic zero and $C$ an algebraically closed field
that contains $Q$. We write $\Aut(C/Q)$ for the group of all automorphisms of
$C$ that act identically on $Q$. It is well known that $Q$ coincides with the
subfield of $\Aut(C/Q)$-invariants in $C$. Let $V$ be a $Q$-vector space of
finite positive dimension $n$ and $V_C=V\otimes_Q C$ the corresponding
$n$-dimensional $C$-vector space. Further we will identify $V$ with the
$Q$-vector subspace $V\otimes 1$ of $V_C$. The group $\Aut(C/Q)$ acts on $V_C$
by $C$-semilinear automorphisms:
$$\sigma(v\otimes c)=v\otimes \sigma(c) \ \forall \ \sigma \in \Aut(C/Q), v \in
V, c\in C.$$ Clearly, $V$ coincides with the $Q$-vector subspace of
$\Aut(C/Q)$-invariants in $V_C$.

The following assertion seems to be  well-known but we were unable to find a
reference. (However, see \cite{KolchinLang}.)

\begin{prop}
Let $\tilde{W}$ be an $\Aut(C/Q)$-invariant $C$-vector subspace of $V_C$. Then
there exists a $Q$-vector subspace $W$ of $V$ such that $\tilde{W}$ coincides
with $W_C=W\otimes_Q C \subseteq V\otimes_Q C=V_C$. In addition, $W$
coincides with the $Q$-subspace of $\Aut(C/Q)$-invariants in $\tilde{W}$.
\end{prop}

\begin{proof}
Let us pick a basis $\{e_1, \dots , e_n\}$ of $V$.
 Let us put $m=\dim_{C}(\tilde{W})$. Clearly, $m \le n$ and we may assume that $m \ge 1$.
If $m=1$ then $\tilde{W}$ contains a vector $w'=\sum_{i=1}^n c_i e_i$ such
that, at least, one of its coordinates say, $c_j$ is $1$. Since
$\tilde{W}=C\cdot w'$ is $\Aut(C/Q)$-invariant, we conclude that all
coordinates $c_i$'s  are $\Aut(C/Q)$-invariant and therefore lie in $Q$. This
means that $w' \in V$ and one may put $W=Q\cdot w'$. On the other hand, if
$n=1$ then $m=1$ and we are also done.

We use induction by $n$. Assume that $1<m$ and consider the $C$-subspace
$\tilde{W}_0$ that is the intersection of $\tilde{W}$ and the hyperplane
$\sum_{i=1}^{n-1} C e_i$.
 Clearly, $\tilde{W}_0$ is the $\Aut(C/Q)$-invariant
$C$-vector subspace in  the $(n-1)$-dimensional $\{\sum_{i=1}^{n-1}Q
e_i\}\otimes_Q C$. By induction assumption, there exists a $Q$-vector subspace
$W_0$ of $\{\sum_{i=1}^{n-1}Q e_i\}\otimes_Q C$ such that
$\tilde{W}_0=W_0\otimes_Q C$. If $\tilde{W}=\tilde{W}_0$ then we are done.
 So, assume $\tilde{W}\ne \tilde{W}_0$.
 Then
$\dim_C(\tilde{W}_0)=\dim_C(\tilde{W})-1=m-1>0$. Since
$\dim_Q({W}_0)=\dim_C(\tilde{W}_0)$, we conclude that $\dim_Q({W}_0)=m-1<n-1$.
 Let us choose a $(n-m)$-dimensional $Q$-vector subspace $W_1$ of
$\{\sum_{i=1}^{n-1}Q e_i\}$ such that $\{\sum_{i=1}^{n-1}Q e_i\}=W_0\oplus
W_1$. We have $V=(W_1 \oplus Q e_n)\oplus W_0, \ W_0\otimes _Q C \subset
\tilde{W}$. Notice that $\dim_{Q}(W_1 \oplus Q e_n)=n-(m-1)<n$. Let us consider
the $C$-vector subspace
 $\tilde{W}_2=\tilde{W} \cap \{(W_1 \oplus
Q e_n) \otimes _Q C\}$. Clearly, $\tilde{W}=\tilde{W}_0\oplus \tilde{W}_2$. By
induction assumption applied to the $\Aut(C/Q)$-invariant $C$-subspace
$\tilde{W}_2$ of $(W_1 \oplus Q e_n) \otimes _Q C$, we conclude that there
exists a $Q$-vector subspace $W_2 \subset W_2 \oplus Q e_n$ such that
$\tilde{W}_2=W_2\otimes_Q C$. Now we may put $W=W_0\oplus W_2$.

\end{proof}


\end{document}